\documentclass[aap]{imsart}

\usepackage{amsmath,amsthm,amssymb,booktabs,mathtools}
\RequirePackage[colorlinks,citecolor=blue,urlcolor=blue]{hyperref}
\RequirePackage{natbib}
\usepackage{dsfont}
\usepackage{enumitem}
\usepackage{hyperref}
\usepackage{graphicx}
\usepackage{caption}

\arxiv{arXiv:1611.03840}
\allowdisplaybreaks

\startlocaldefs
\newtheorem{theorem}{Theorem}[section]
\newtheorem{proposition}[theorem]{Proposition}
\newtheorem{definition}[theorem]{Definition}
\newtheorem{lemma}[theorem]{Lemma}
\newtheorem{corollary}[theorem]{Corollary}
\newtheorem*{remark*}{Remark}

\newcommand{\ccirc}{\mathbin{\mathchoice
  {\xcirc\scriptstyle}
  {\xcirc\scriptstyle}
  {\xcirc\scriptscriptstyle}
  {\xcirc\scriptscriptstyle}
}}
\newcommand{\xcirc}[1]{\vcenter{\hbox{$#1\circ$}}}

\newcommand{\p}{\mathbb{P}}

\newcommand{\m}{\mu_{n, q}}
\newcommand{\n}{\mu_{n, q_n}}
\newcommand{\inv}{l}
\newcommand{\IV}{\text{Inv}}

\newcommand{\lis}{\text{LIS}}
\newcommand{\lds}{\text{LDS}}
\newcommand{\lcs}{\text{LCS}}
\newcommand{\bs}{\boldsymbol}
\endlocaldefs

\begin{document}

\begin{frontmatter}

\title{The Length of the Longest Common Subsequence of Two Independent Mallows Permutations}
\runtitle{Longest Common Subsequence of Mallows Permutations}


\author{\fnms{Ke} \snm{Jin}\corref{}\ead[label=e1]{kejin@udel.edu}\thanksref{t1}}
\thankstext{t1}{Supported by NSF grant DMS-1261010 and Sloan Research Fellowship}
\address{Department of Mathematical Sciences, University of Delaware\\ \printead{e1}}
\affiliation{University of Delaware}

\runauthor{Ke Jin}

\begin{abstract}
The Mallows measure is a probability measure on $S_n$ where the probability of a permutation $\pi$ is proportional to $q^{l(\pi)}$ with $q > 0$ being a parameter and $l(\pi)$ the number of inversions in $\pi$. We prove a weak law of large numbers for the length of the longest common subsequences of two independent permutations drawn from the Mallows measure, when $q$ is a function of $n$ and $n(1-q)$ has limit in $\mathbb{R}$ as $n \to \infty$.
\end{abstract}



\end{frontmatter}


\section{Introduction}\label{S1}

\subsection{Background}
The longest common subsequence(LCS) problem is a classical problem which has application in fields such as molecular biology (see, e.g., \cite{pevzner}) , data comparison and software version control. Most previous works on the LCS problem are focused on the case when the strings are generated uniformly at random from a given alphabet. Notably, Chv\'{a}tal and Sankoff \cite{Sankoff} proved that the expected length of the LCS of two random $k$-ary sequences of length $n$ when normalized by $n$ converges to a constant $\gamma_k$. Since then, various endeavors \cite{Deken, Dancik, Paterson, Lueker} have been made to determine the value of $\gamma_k$. The exact values of $\gamma_k$ are still unknown. The known lower and upper bounds \cite{Lueker} for $\gamma_2$ are
\[
0.788071 < \gamma_2 < 0.826280.
\]
In contrast to the LCS of two random strings, the LCS of two permutations is well connected to the longest increasing subsequence(LIS) problem (cf.\,Proposition 3.1 in \cite{Houdre}). This can be seen from the following two facts,
\begin{itemize}
\item For any $\pi \in S_n$, the length of the LCS of $\pi$ and the identity in $S_n$ is equal to the length of the LIS of $\pi$.
\item For any $\pi, \tau\in S_n$, the length of the LCS of $\pi$ and $\tau$ is equal to the length of the LCS of $\tau^{-1}\ccirc\pi$ and the identity in $S_n$.
\end{itemize}
From the above two properties, it is easily seen that, if $\pi, \tau$ are independent and either $\pi$ or $\tau$ is uniformly distributed on $S_n$ the length of the LCS of $\pi$ and $\tau$ has the same distribution as the length of the LIS of a uniformly random permutation. The length of the LIS of a uniformly random permutation has been well studied with major contributions from Hammersley \cite{Hammersley}, Logan and Shepp \cite{Logan}, Vershik and Kerov \cite{VK} and culminating with the groundbreaking work of Baik, Deift and Johansson \cite{Baik} who prove that, under proper scaling, the length of the LIS converges to the Tracy-Widom distribution. Therefore, the length of the LCS of two permutations is only of interest when both permutations are non-uniformly distributed.
In this paper we study the length of the LCS of two independent permutations drawn from the Mallows measure. 
\begin{definition}\label{DD}
Given $\pi \in S_n$, the inversion set of $\pi$ is defined by
\[
\IV(\pi) \coloneqq \{ (i, j) : 1 \le i < j \le n \text{ and } \pi(i) > \pi(j) \},
\]
and the inversion number of $\pi$, denoted by $\inv(\pi)$, is defined to be the cardinality of $\IV(\pi)$.
\end{definition}
The Mallows measure on $S_n$ is introduced by Mallows in \cite{mallows1957}. For $q > 0$, the $(n, q)$ - Mallows measure on $S_n$ is given by
\[
\m(\pi) \coloneqq \frac{q^{\inv(\pi)}}{Z_{n, q}},
\]
where $Z_{n, q}$ is the normalizing constant. In other words, under the Mallows measure with parameter $q>0$, the probability of a permutation $\pi$ is proportional to $q^{\inv(\pi)}$. Mallows measure has been used in modeling ranked and partially ranked data (see, e.g., \cite{CR, FV, Ma}).
\begin{definition}
For any $\pi, \tau \in S_n$, define the length of the longest common subsequence of $\pi$ and $\tau$ as follows,
\begin{align*}
\lcs(\pi, \tau) \coloneqq \max(m : \exists\,i_1 < \cdots <& i_m \text{ and } j_1 < \cdots < j_m \\
&\text{ such that } \pi(i_k) = \tau(j_k)\text{ for all } k \in [m]).
\end{align*}
\end{definition}
Given the close connection of the LCS of two permutations and the LIS problem, to prove our results, we are able to make use of  the techniques developed in \cite{MuellerStarr, Naya} in which weak laws of large numbers of the length of the LIS of permutation under Mallows measure have been proven for different regimes of $q$.

\subsection{Results}

Before stating the main theorem, we introduce the following lemma proved in \cite{KeEm}, which shows the convergence of the empirical measure of a collection of random points defined by two independent Mallows permutations.
\begin{lemma}\label{L100}
Suppose that $\{q_n\}_{n=1}^{\infty}$ and $\{q'_n\}_{n=1}^{\infty}$ are two sequences such that $\lim_{n\to\infty} n (1-q_n) = \beta$  and
$\lim_{n\to\infty} n (1-q'_n) = \gamma$, with $\beta$, $\gamma \in \mathbb{R}$. Let $\p_n$ denote the probability measure on $S_n \times S_n$
such that $\p_n\big((\pi, \tau)\big) = \n(\pi) \cdot \mu_{n, q'_n}(\tau)$, i.~e.~$\p_n$ is the product measure of $\n$ and $\mu_{n, q'_n}$.
For any $R = (x_1, x_2]\times (y_1, y_2] \subset [0, 1]\times [0, 1]$, we have
\begin{equation}
\lim_{n \to \infty} \p_n \left(\left|\,\frac{1}{n}\sum_{i = 1}^{n} \mathds{1}_R\left(\frac{\pi(i)}{n}, \frac{\tau(i)}{n}\right) - \int_{R} \rho(x, y)\, dxdy\, \right| > \epsilon \right) = 0,  \label{eq:L100a}
\end{equation}
for any $\epsilon > 0$, with
\begin{equation}
\rho(x, y) \coloneqq \int_{0}^{1} u(x, t, \beta)\cdot u(t, y, \gamma)\,dt, \label{eq:L100z}
\end{equation}
where
\begin{equation}
u(x,y, \beta) \coloneqq \frac{(\beta/2) \sinh(\beta/2)}
{\left(e^{\beta/4} \cosh(\beta[x-y]/2)-e^{-\beta/4}\cosh(\beta[x+y-1]/2)\right)^2}, \label{eq:M1}
\end{equation}
for $\beta \neq 0$, and $u(x, y, 0) \coloneqq 1$.
\end{lemma}
The density $u(x, y, \beta)$ in (\ref{eq:M1}), obtained by Starr in \cite{Starr}, is the limiting distribution of the empirical measure induced by Mallows permutation when the parameters $q_n$ satisfy that $\lim_{n \to \infty} n (1-q_n) = \beta$. The limiting distribution of the points of a random permutation is known as a \textit{permuton} (cf.\,\cite{hoppen}) and has recently been studied in the context of finding the limiting distribution of permutation statistics such as cycle lengths \cite{mukherjee} and certain limit shapes of permutations with fixed pattern densities \cite{kenyon}.

The main result of this paper is a weak law of large numbers of the LCS of two permutations drawn independently from the Mallows measure. The first observation, which is proved in Corollary \ref{C33}, is that the length of LCS of two permutations $\pi$ and $\tau$ is equal to the length of the longest increasing points in the collection of points
\[
\bs{z}(\pi^{-1},\tau^{-1}) \coloneqq \left\{\left( 
 \frac{\pi^{-1}(i)}{n}, \frac{\tau^{-1}(i)}{n}\right) \right\}_{i \in [n]}.
\]
At a high level, our proof follows the approach of Deuschel and Zeitouni \cite{DZ95} who showed weak laws for the LIS of i.i.d.\,points drawn according to some density in a box. They partition the box into a grid of smaller boxes whose size is chosen to be such that the distribution of points within them is close to uniform. The weak law for the LIS of uniformly random permutations \cite{VK} can be applied to points in these boxes to estimate the number of increasing points in the neighborhood of any increasing path. In our case, this approach fails because the points in the box are no longer i.i.d..

Indeed, in a prior work, Mueller and Starr \cite{MuellerStarr} applied Deuschel and Zeitouni's approach to show a weak law for the LIS of a Mallows permutation, where, due to properties of Mallows measure, the permutation induced by the points in a smaller box is also Mallows distributed. They coupled the distribution of points to two i.i.d.\,point processes to overcome this problem.  In our case, this does not seem to be applicable directly, since the induced permutation by the points in a box is no longer Mallows or the product of independent Mallows permutations. We follow a different approach. We prove a combinatorial fact using the properties of the weak Bruhat order, to say that the distribution of the LIS of points in a small box can be stochastically bounded between the LIS and the LDS of a Mallows permutation restricted to a certain fixed set of indices. In their work, Mueller and Starr derived estimates on the LIS of a Mallows permutation in a small box, however we cannot use these estimates directly because of the restriction to an arbitrary set of indices. To overcome this, we generalize their estimates to the LIS of a Mallows permutation restricted to an arbitrary set of indices. Our argument recovers their result for small enough $\beta$ (which is the relevant case) and gives a slightly more streamlined proof.

Specifically, we establish two results to apply the approach above. The first result, deduced from Lemma \ref{L100}, is that the number of points $\bs{z}(\pi^{-1},\tau^{-1})$ contained in any fixed rectangle, when divided by the size of the permutation, converges in probability to a constant. The second result, proved in Lemma \ref{L501}, is that the length of the longest increasing points in $\bs{z}(\pi^{-1},\tau^{-1})$ within a small box $R$ is close to the size of the LIS in the uniform case, i.e., it is approximately $2\sqrt{|\bs{z}(\pi^{-1},\tau^{-1})\cap R|}$. The main theorem is the following.
\begin{theorem}\label{M3}
Let $B^1_{\nearrow}$ denote the set of nondecreasing, $C^1_b$ functions $\phi : [0, 1] \rightarrow [0, 1]$, with $\phi(0) = 0$ and $\phi(1) = 1$, where $C^1_b$ denotes the set of functions which have bounded and continuous first order derivative. Define function $J : B^1_{\nearrow} \rightarrow \mathbb{R}$,
\[
J(\phi) \coloneqq \int_0^1 \sqrt{\dot{\phi}(x) \rho(x, \phi(x))} \, dx, \quad  \text{ and } \quad \bar{J} \coloneqq \sup_{\phi \in B^1_{\nearrow}} J(\phi),
\]
where $\rho(x, y)$ is the density defined in (\ref{eq:L100z}) and $\dot{\phi}$ denotes the derivative of $\phi$. Under the same conditions as in Lemma \ref{L100}, for any $\epsilon > 0$, we have
\begin{equation}
\lim_{n \to \infty} \p_n \left(\left|\,\frac{\lcs(\pi, \tau)}{\sqrt{n}} - 2\bar{J} \, \right| < \epsilon \right) = 1. \label{eq:M3}
\end{equation}

\end{theorem}

Finally, we derive the limiting constant in the special case when $\beta = \gamma$.

\begin{corollary}\label{C92}
Suppose that $\{q_n\}_{n=1}^{\infty}$ and $\{q'_n\}_{n=1}^{\infty}$ are two sequences such that $\lim_{n\to\infty} n (1-q_n) = \lim_{n\to\infty} n (1-q'_n) = \beta$ with $\beta \neq 0$. Then, the constant $\bar J$ in Theorem \ref{M3} is given by
\[
\bar{J} = {\textstyle\sqrt{\frac{\beta}{6 \sinh{(\beta/2)}}}}\cdot\int_0^1 \sqrt{ \cosh{(\beta/2)} + 2 \cosh{\big(\beta[2x - 1]/2\big)}}\, dx.
\]
\end{corollary}

\section{Reducing LCS problem to LIS problem}

\begin{definition}\label{D3}
Given a set of points in $\mathbb{R}^2$: $\bs{z} = \{z_1, z_2, \ldots, z_n\}$, where $z_i = (x_i, y_i) \in \mathbb{R}^2$, we say that $(z_{i_1}, z_{i_2}, \ldots, z_{i_m})$ is an increasing subsequence if
\[
x_{i_j} < x_{i_{j+1}}, \quad y_{i_j} < y_{i_{j+1}}, \quad j = 1, \ldots, m-1.
\]
Above we do not require $i_j < i_{j+1}$.
Let $\lis(\bs{z})$ denote the length of the longest increasing subsequence of $\bs{z}$.
\end{definition}
\begin{definition}\label{D31}
Given $\bs{a} = (a_1, \ldots, a_n) \in \mathbb{R}^n$, $\bs{b} = (b_1, \ldots, b_n)~\in~\mathbb{R}^n$, we say that $((a_{i_1}, b_{i_1}), (a_{i_2}, b_{i_2}), \ldots, (a_{i_m}, b_{i_m}))$ is an increasing subsequence between $\bs{a}$ and $\bs{b}$ if
\[
a_{i_j} < a_{i_{j+1}}, \quad b_{i_j} < b_{i_{j+1}}, \quad j = 1, \ldots, m-1.
\]
Above we do not require $i_j < i_{j+1}$.
Let $\lis(\bs{a}, \bs{b})$ denote the length of the longest increasing subsequence between $\bs{a}$ and $\bs{b}$.
Let $\lis(\bs{a}) \coloneqq \lis(id, \bs{a})$, $\lds(\bs{a}) \coloneqq \lis(id^r, \bs{a})$, where $id = (1, 2, \ldots, n)$ denotes the identity in $S_n$ and $id^r = (n, \ldots, 1)$ denotes the reversal of identity in $S_n$. Hence $\lis(\bs{a})$ is the length of the longest increasing subsequence of $\bs{a}$ and $\lds(\bs{a})$ is the length of the longest decreasing subsequence of $\bs{a}$.
\end{definition}
Note that Definition \ref{D31} allows us to define $\lis(\pi,\tau)$, the length of the longest increasing subsequence of two permutations, by regarding $\pi$ and $\tau$ as vectors in $\mathbb{Z}^n$.
We show that $\lcs(\pi, \tau) = \lis(\pi^{-1}, \tau^{-1})$, which allows us to reduce the LCS problem to an LIS problem.

\begin{lemma}\label{L32}
Given $\pi, \tau \in S_n$, we have
\[
\lcs(\pi, \tau) = \lcs(\sigma \ccirc \pi, \sigma \ccirc \tau), \qquad \lis(\pi, \tau) = \lis(\pi \ccirc \sigma, \tau \ccirc \sigma),
\]
for any $\sigma \in S_n$.
\end{lemma}
\begin{proof}
Suppose $(a_1, a_2, \ldots, a_m)$ is a common subsequence of $\pi$ and $\tau$, then $(\sigma(a_1), \sigma(a_2), \ldots, \sigma(a_m))$ is a common subsequence of $\sigma \ccirc \pi$ and $\sigma \ccirc \tau$. Hence,
\[
\lcs(\pi, \tau) \le \lcs(\sigma \ccirc \pi, \sigma \ccirc \tau) \le
\lcs(\sigma^{-1} \ccirc \sigma\ccirc \pi, \sigma^{-1} \ccirc \sigma \ccirc \tau) = \lcs(\pi, \tau).
\]
Similarly, suppose $((\pi(i_1), \tau(i_1)),\ (\pi(i_2), \tau(i_2)),\ \ldots,\ (\pi(i_m), \tau(i_m)))$ is an increasing subsequence between $\pi$ and $\tau$, then
$((\pi\ccirc \sigma(i'_1), \tau\ccirc \sigma(i'_1)),\newline (\pi\ccirc \sigma(i'_2), \tau\ccirc \sigma(i'_2)),\ \ldots,\ (\pi\ccirc \sigma(i'_m), \tau\ccirc \sigma(i'_m)))$ is an increasing subsequence between $\pi \ccirc \sigma$ and $\tau \ccirc \sigma$, where $i'_k =
\sigma^{-1}(i_k)$ for $k \in [m]$. Hence,
\[
\lis(\pi, \tau) \le \lis(\pi \ccirc \sigma, \tau \ccirc \sigma) \le \lis(\pi \ccirc \sigma \ccirc \sigma^{\-1}, \tau \ccirc \sigma \ccirc \sigma^{-1}) = \lis(\pi, \tau). \qedhere
\]
\end{proof}

\begin{corollary}\label{C33}
For any $\pi, \tau \in S_n$, $\lcs(\pi, \tau) = \lis(\pi^{-1}, \tau^{-1})$.
\end{corollary}
\begin{proof}
By the previous lemma, we have
\[
\lcs(\pi, \tau) = \lcs(id, \pi^{-1}\ccirc\tau) = \lis(id, \pi^{-1}\ccirc\tau) = \lis(\tau^{-1}, \pi^{-1})
\]
In the second equality, we use the following trivial fact,
\[
\lcs(id, \pi) = \lis(\pi) = \lis(id, \pi). \qedhere
\]
\end{proof}

\section{Weak Bruhat order}
Before introducing the weak Bruhat order, we make the following definition.

\begin{definition}
Given $\pi \in S_n$ and $\bs{a} = (a_1, a_2, \ldots, a_k)$, where $a_i \in [n]$ and $a_1 < a_2 < \cdots < a_k$, let
$\pi(\bs{a}) = ( \pi(a_1), \pi(a_2), \ldots, \pi(a_k))$. Let $\pi_{\bs{a}} \in S_k$ denote the permutation induced by $\pi(\bs{a})$, i.\,e.\,$\pi_{\bs{a}}(i) = j$ if $\pi(a_i)$ is the $j$-th smallest term in $\pi(\bs{a})$.
\end{definition}

Lemma \ref{L501} says that the LIS of the points $\{(\frac{\pi(i)}{n}, \frac{\tau(i)}{n})\}_{i\in[n]}$ that fall in a small box is close to the uniform case. Let $\bs{a} = (a_1, \ldots, a_k)$ be an increasing sequence of indices with $a_i \in [n]$. To prove Lemma \ref{L501}, we will show that there exists a coupling of permutations $(X,Y, X', X'')$, where $X,X'$ and $X''$ are distributed according to $\m$ and $Y$ is independent of $X$ with an arbitrary distribution on $S_n$. Under this coupling $\lis(X_{\bs{a}}, Y_{\bs{a}})$ will be bounded by $\lis(X'_{\bs{a}})$ and $\lds(X''_{\bs{a}})$. The main tool we use to construct the coupling is the weak Bruhat order on $S_n$. 

Recall that for a permutation $\pi \in S_n$, $l(\pi)$ denotes the number of inversions of $\pi$ and $\IV(\pi)$ denotes the set of inversions of $\pi$ as defined in Definition~\ref{DD}. Let $(i, j)$ denote the transposition in $S_n$ and $s_i \coloneqq (i, i+1)$ the adjacent transposition in $S_n$.

\begin{definition}
The left weak Bruhat order $(S_n, \le_L)$ is defined as the transitive closure of the relations
\[
\pi \le_L \tau \quad \text{if} \quad \tau = s_i \ccirc \pi \ \text{ and }\  \inv(\tau) = \inv(\pi) + 1.
\]
\end{definition}

We are multiplying permutations right-to-left. For instance, $s_2 \ccirc 2413 = 3412$. 
The \emph{right weak Bruhat order} $(S_n, \le_R)$  is defined in the same way except that the covering relationship is given by $\tau = \pi \ccirc s_i$ and $\inv(\tau) = \inv(\pi) + 1$.

One characterization of the left weak order is the following proposition(cf. \cite{abello}). We provide its proof here for the completeness of the paper.

\begin{proposition}\label{abe}
\[
\pi \le_L \tau \quad \text{if and only if} \quad \IV(\pi) \subseteq \IV(\tau).
\]
\end{proposition}
\begin{proof}
Suppose $\tau$ covers $\pi$, i.e., $s_i \ccirc \pi = \tau$ and $\inv(\pi) + 1 = \inv(\tau)$. It is easy to see that $\IV(\tau) = \IV(\pi)\cup\{\big(\pi^{-1}(i), \pi^{-1}(i+1)\big)\}$. For arbitrary $\pi$ and $\tau$, $\pi \le_L \tau$ implies that there exists a sequence of permutations $\{\sigma_0, \ldots, \sigma_k\}$ such that $\sigma_{i+1}$ covers $\sigma_{i}$ and $\pi = \sigma_0 \le_L \cdots \le_L \sigma_k = \tau$. Hence $\pi \le_L \tau$ implies $\IV(\pi) \subseteq \IV(\tau)$. On the other hand, given $\IV(\pi) \subseteq \IV(\tau)$, to show $
\pi \le_L \tau$ it suffices to show that there exists an adjacent transposition $s_i$ such that $\IV(\pi) \subseteq \IV(s_i\ccirc\tau) \subset\IV(\tau)$. Let $k$ be the smallest $i$ such that $\pi^{-1}(i) \neq \tau^{-1}(i)$. Let $j = \pi^{-1}(k)$ and $h = \tau(j)$. Since $h > k \ge 1$, define $j' =  \tau^{-1}(h-1)$. By the choice of $k$, we have $\pi(j') > k$. It follows that $j < j'$, since otherwise $(j', j) \in \IV(\pi)$ and $(j', j) \notin \IV(\tau)$. Therefore we have $\IV(\pi) \subseteq \IV(s_{h-1}\ccirc\tau) \subset\IV(\tau)$.
\end{proof}

\begin{lemma}\label{L34}
Given $\pi, \tau \in S_k$ with $\pi \le_L \tau$, for any $n \ge k$, $0 < q \le 1$ and increasing indices $\bs{a} = ( a_1, a_2, \ldots, a_k )$ with $a_i \in [n]$, there exists a coupling $(X, Y)$ such that $X \sim \m$, $Y \sim \m$ and
\[
\lis(X_{\bs{a}}, \pi) \ge \lis(Y_{\bs{a}}, \tau).
\]
\end{lemma}
\begin{proof}
First, we claim that it suffices to show the case when $\tau$ covers $\pi$ in $(S_k, \le_L)$, that is $\inv(\tau) = \inv(\pi) +1$ and $\tau = s_i \ccirc \pi$ for some $i \in [k-1]$. The claim can be shown by induction on the Kendall's tau distance of $\pi$ and $\tau$, i.e., the minimum number of adjacent transpositions multiplied to $\pi$ from the left to get $\tau$. Suppose we have $\pi \le_L \sigma \le_L \tau$ in $S_k$ with $\inv(\pi) < \inv(\sigma) < \inv(\tau)$. By the induction hypothesis there exist two couplings $(X, Y)$ and $(Y', Z)$, which are not necessarily defined in the same probability space, such that $X, Y, Y', Z$ have the same marginal distribution $\m$ and
\begin{equation}
\lis(X_{\bs{a}}, \pi) \ge \lis(Y_{\bs{a}}, \sigma), \qquad \lis(Y'_{\bs{a}}, \sigma) \ge \lis(Z_{\bs{a}}, \tau). \label{eq:L34a}
\end{equation}
We can construct a new coupling $(X', Z')$ as follows,
\begin{enumerate}
\item[(1)] Sample a permutation $\xi \in S_n$ according to the distribution $\m$.
\item[(2)] Sample $X'$ according to the induced distribution on $S_n$ by the first coupling $(X, Y)$ conditioned on $Y = \xi$.
\item[(3)] Sample $Z'$ according to the induced distribution on $S_n$ by the second coupling $(Y', Z)$ conditioned on $Y' = \xi$.
\end{enumerate}
By the law of total probability, it is easily seen that $X' \sim \m$ and $Z' \sim \m$. Also, regardless of which permutation $\xi$ being sampled in the first step, by (\ref{eq:L34a}), we have
\[
\lis(X'_{\bs{a}}, \pi) \ge \lis(\xi_{\bs{a}}, \sigma) \ge \lis(Z'_{\bs{a}}, \tau). 
\]
In the remainder of the proof, we assume $\tau = s_i \ccirc \pi$ and $\inv(\tau) = \inv(\pi) +1$. Note that, for any $\sigma \in S_n$,
\begin{equation}
\sigma \ccirc (i, j) = (\sigma(i), \sigma(j))\ccirc \sigma, \qquad \sigma_{\bs{a}}\ccirc (i, j) = (\sigma \ccirc (a_i, a_j))_{\bs{a}}. \label{eq:L34b}
\end{equation}
Let $r = a_{\pi^{-1}(i)}$ and $t = a_{\pi^{-1}(i+1)}$. Since $\inv(\tau) = \inv(\pi) +1$, we have $\pi^{-1}(i) < \pi^{-1}(i+1)$, thus, $r < t$. Let $A \coloneqq \{\{\sigma, \sigma \ccirc (r, t)\} : \sigma \in S_n \text{ and } \sigma(r) < \sigma(t) \}$. Clearly, $A$ is a partition of $S_n$. Then we construct the coupling $(X, Y)$ as follows:
\begin{enumerate}
\item[(1)] Choose a set in $A$ according to measure $\m$, i.\,e.\, the set $\{\sigma, \sigma \ccirc (r, t)\}$ is chosen with probability
$\m(\{\sigma, \sigma \ccirc (r, t)\})$.
\item[(2)] Suppose the set $\{\sigma, \sigma \ccirc (r, t)\}$, with $\sigma(r) < \sigma(t)$, is chosen in the first step. Flip a coin with probability of heads being
\[
p = \frac{q^{\inv(\sigma)} - q^{\inv(\sigma \ccirc (r, t))}}{q^{\inv(\sigma)} + q^{\inv(\sigma \ccirc (r, t))}}.
\]
Note that the probability of heads $p$ is nonnegative because we have $0 < q \le 1$ and the following fact:
\[
i < j \text{ and } \sigma(i) < \sigma(j)\ \Rightarrow\ \inv(\sigma) < \inv(\sigma \ccirc (i, j)), \quad \forall \sigma \in S_n.
\]
\item[(3)] If the outcome is head, then we set $X = Y = \sigma$.
\item[(4)] If the outcome is tail, then, with equal probability, we set either $X = \sigma$, $Y = \sigma \ccirc (r, t)$ or $X = \sigma \ccirc (r, t)$, $Y = \sigma$.
\end{enumerate}
It can be verified that $(X, Y)$ thus defined has the correct marginal distribution $\m$. In the following we show that
\begin{equation}
\lis(X_{\bs{a}} \ccirc \pi^{-1}) \ge \lis(Y_{\bs{a}} \ccirc \tau^{-1}). \label{eq:L34c}
\end{equation}
Then, the lemma follows by Lemma \ref{L32} because
\begin{align*}
\lis(X_{\bs{a}} \ccirc \pi^{-1}) &= \lis(X_{\bs{a}} \ccirc \pi^{-1}, id) = \lis(X_{\bs{a}}, \pi),\\
\lis(Y_{\bs{a}} \ccirc \tau^{-1}) &= \lis(Y_{\bs{a}} \ccirc \tau^{-1}, id) = \lis(Y_{\bs{a}}, \tau).
\end{align*}
Suppose the set $\{\sigma, \sigma \ccirc (r, t)\}$, with $\sigma(r) < \sigma(t)$, is chosen in the first step. If the outcome in the second step is tail, we verify that $X_{\bs{a}} \ccirc \pi^{-1} = Y_{\bs{a}} \ccirc \tau^{-1}$. When $X = \sigma$, $Y = \sigma \ccirc (r, t)$, by
(\ref{eq:L34b}), we have
\begin{align*}
X_{\bs{a}} \ccirc \pi^{-1} &= \sigma_{\bs{a}} \ccirc \pi^{-1},\\
Y_{\bs{a}} \ccirc \tau^{-1} &= (\sigma \ccirc (r, t))_{\bs{a}} \ccirc \pi^{-1} \ccirc s_i\\
                            &= (\sigma \ccirc (r, t))_{\bs{a}} \ccirc (\pi^{-1}(i), \pi^{-1}(i+1)) \ccirc \pi^{-1}\\
                            &= (\sigma \ccirc (r, t) \ccirc (r, t))_{\bs{a}} \ccirc \pi^{-1}\\
                            &= \sigma_{\bs{a}} \ccirc \pi^{-1}.
\end{align*}
When $X = \sigma \ccirc (r, t)$, $Y = \sigma$, again by (\ref{eq:L34b}), we have
\begin{align*}
X_{\bs{a}} \ccirc \pi^{-1} &= (\sigma \ccirc (r, t))_{\bs{a}} \ccirc \pi^{-1}\\
                           &= \sigma_{\bs{a}} \ccirc (\pi^{-1}(i), \pi^{-1}(i+1)) \ccirc \pi^{-1}\\
                           &= \sigma_{\bs{a}} \ccirc \pi^{-1} \ccirc s_i,\\
Y_{\bs{a}} \ccirc \tau^{-1} &= \sigma_{\bs{a}} \ccirc \pi^{-1} \ccirc s_i.
\end{align*}
If the outcome in the second step is head, we have 
\[
X_{\bs{a}} \ccirc \pi^{-1} = \sigma_{\bs{a}} \ccirc \pi^{-1} \quad\text{ and }\quad
Y_{\bs{a}} \ccirc \tau^{-1} = \sigma_{\bs{a}} \ccirc \pi^{-1} \ccirc s_i.
\]
Since $\sigma(r) < \sigma(t)$, i.\,e.\,, $\sigma(a_{\pi^{-1}(i)}) < \sigma(a_{\pi^{-1}(i+1)})$, we have $\sigma_{\bs{a}} \ccirc \pi^{-1}(i) < \sigma_{\bs{a}} \ccirc \pi^{-1}(i+1)$. Hence $Y_{\bs{a}} \ccirc \tau^{-1}$ covers $X_{\bs{a}} \ccirc \pi^{-1}$ in $(S_k, \le_R)$. (\ref{eq:L34c}) follows.
\end{proof}

\begin{remark*}
A special case of Lemma \ref{L34} is when $k = n$, in which the only choice for $\bs{a}$ is the vector $(1, 2, \ldots, n)$ whence $X_{\bs{a}} = X$, $Y_{\bs{a}} = Y$.
\end{remark*}

In Lemma \ref{L341}, we prove a similar result for the case when $q \ge 1$, using the following property of Mallows permutations (cf.\,Lemma 2.2 in \cite{Naya}).
\begin{proposition}\label{P3}
For any $n \ge 1$ and $q > 0$, if $\pi \sim \m$ then $\pi^r \sim \mu_{n, 1/q}$ and $\pi^{-1} \sim \m$.
\end{proposition}

\begin{lemma}\label{L341}
Given $\pi, \tau \in S_k$ with $\pi \le_L \tau$, for any $n \ge k$, $q \ge 1$ and increasing indices $\bs{a} = ( a_1, a_2, \ldots, a_k )$ with $a_i \in [n]$, there exists a coupling $(X, Y)$ such that $X \sim \m$, $Y \sim \m$ and
\[
\lis(X_{\bs{a}}, \pi) \le \lis(Y_{\bs{a}}, \tau).
\]
\end{lemma}
\begin{proof}
Given $\pi \in S_n$, recall that $\pi^r$ denote the reversal of $\pi$. For any $\pi \in S_n$, we have $\IV(\pi^r) =
\{(i, j): 1\le i < j \le n \text{ and } (n+1-j, n+1-i) \notin \IV(\pi)\}$. Hence, $\pi \le_L \tau$ implies $\tau^r \le_L \pi^r$. By Lemma \ref{L34}, there exists a coupling $(U, V)$ such that $U \sim \mu_{n, 1/q}$, $V \sim \mu_{n, 1/q}$ and
\[
\lis(U_{\bs{a}'}, \pi^r) \le \lis(V_{\bs{a}'}, \tau^r),
\]
where $\bs{a}' = (a'_1, \ldots, a'_k)$ with $a'_i = n + 1 - a_{k+1-i}$.
Define $(X, Y) \coloneqq (U^r, V^r)$. By Proposition \ref{P3}, $X \sim \m$, $Y \sim \m$. Moreover, we have
\begin{align*}
\lis(X_{\bs{a}}, \pi) &= \lis((X_{\bs{a}})^r, \pi^r) = \lis((X^r)_{\bs{a}'}, \pi^r) = \lis(U_{\bs{a}'}, \pi^r)\\
                      &\le \lis(V_{\bs{a}'}, \tau^r) = \lis((Y^r)_{\bs{a}'}, \tau^r) = \lis((Y_{\bs{a}})^r, \tau^r)\\
                      &= \lis(Y_{\bs{a}}, \tau). \qedhere
\end{align*}
\end{proof}


\begin{lemma}\label{L360}
Given increasing indices $\bs{a} = ( a_1, a_2, \ldots, a_k )$ with $a_i \in[n]$, for any $0 < q \le 1$ and any distribution $\nu$ on $S_k$, there exists a coupling $(X, Y, Z)$ such that the following holds.
\begin{itemize}
\item[(a)] $X$ and $Y$ are independent.
\item[(b)] $X \sim \m$, $Y \sim \nu$ and $Z \sim \m$.
\item[(c)] $\lis(X_{\bs{a}}, Y) \le \lis(Z_{\bs{a}})$.
\end{itemize}
\end{lemma}

\begin{proof}
Let $id_{k}$ denote the identity in $S_k$. By the definition of weak Bruhat order, for any $\xi \in S_k$, we have $id_{k} \le_{L} \xi$. Hence, given $\xi \in S_k$, by Lemma \ref{L34}, there exists a coupling $(U, V)$ such that $U \sim \m$, $V \sim \m$ and
$\lis(U_{\bs{a}}, \xi) \le \lis(V_{\bs{a}}, id_{k}) = \lis(V_{\bs{a}})$. Then we construct the coupling $(X, Y, Z)$ as follows.
\begin{itemize}
\item Sample $Y$ according to the distribution $\nu$.
\item Conditioned on $Y = \xi$, $(X, Z)$ has the same distribution as $(U, V)$ defined above.
\end{itemize}
First, we point out that $X$ and $Y$ are independent. Since whatever value $Y$ takes, the conditional distribution of $X$ is $\m$. Moreover, it can be seen that $X$, $Y$ and $Z$ thus defined have the correct marginal distributions. Finally, (c) holds by the construction of the coupling.
\end{proof}

We can prove a similar result for the case when $q \ge 1$.

\begin{lemma}\label{L361}
Given $\bs{a} = (a_1, a_2, \ldots, a_k)$, where $a_1 < \cdots < a_k$ and $a_i \in[n]$, for any $q \ge 1$ and any distribution $\nu$ on $S_k$, there exists a coupling $(X, Y, Z)$ such that the following holds.
\begin{itemize}
\item[(a)] $X$ and $Y$ are independent.
\item[(b)] $X \sim \m$, $Y \sim \nu$ and $Z \sim \m$.
\item[(c)] $\lis(X_{\bs{a}}, Y) \ge \lis(Z_{\bs{a}})$.
\end{itemize}
\end{lemma}

\begin{proof}
The lemma follows by the same argument as in the proof of Lemma~\ref{L360} except that here we use Lemma \ref{L341} instead of Lemma \ref{L34}.
\end{proof}

\begin{lemma}\label{L370}
Given $\bs{a} = (a_1, a_2, \ldots, a_k)$, where $a_1 < \cdots < a_k$ and $a_i \in[n]$. Define $\bar{\bs{a}} \coloneqq (n+1 - a_k, n+1 - a_{k-1}, \ldots, n+1 - a_1)$. For any $0 < q \le 1$ and any distribution $\nu$ on $S_k$, there exists a coupling $(X, Y, Z)$ such that the following holds.
\begin{itemize}
\item[(a)] $X$ and $Y$ are independent.
\item[(b)] $X \sim \m$, $Y \sim \nu$ and $Z \sim \mu_{n, 1/q}$.
\item[(c)] $\lis(X_{\bs{a}}, Y) \ge \lis(Z_{\bar{\bs{a}}})$.
\end{itemize}
\end{lemma}

\begin{proof}
Recall that $\pi^r$ denotes the reversal of $\pi$. If $\pi \sim \nu$, we use $\nu^r$ to denote the distribution of $\pi^r$. Clearly, $\nu = (\nu^r)^r$. By Lemma \ref{L361}, there exists a coupling $(U, V, Z)$ such that
\begin{itemize}
\item $U$ and $V$ are independent.
\item $U \sim \mu_{n, 1/q}$, $V \sim \nu^r$ and $Z \sim \mu_{n, 1/q}$.
\item $\lis(U_{\bar{\bs{a}}}, V) \ge \lis(Z_{\bar{\bs{a}}})$.
\end{itemize}
Define $X \coloneqq U^r$ and $Y \coloneqq V^r$. We have
\begin{align*}
\lis\big(U_{\bar{\bs{a}}}, V\big) &= \lis\Big(\big\{\big(U_{\bar{\bs{a}}}(i), V(i)\big)\big\}_{i\in [k]}\Big)\\
                                &= \lis\Big(\big\{\big((U_{\bar{\bs{a}}})^r(i), V^r(i)\big)\big\}_{i\in [k]}\Big)\\
                                &= \lis\Big(\big\{\big((U^r)_{\bs{a}}(i), V^r(i)\big)\big\}_{i\in [k]}\Big)\\
                                &= \lis\Big(\big\{\big(X_{\bs{a}}(i), Y(i)\big)\big\}_{i\in [k]}\Big)\\
                                &= \lis(X_{\bs{a}}, Y). \qedhere
\end{align*}
\end{proof}

\section{Proof of Theorem \ref{M3}}

We start this section by introducing the following lemma which is analogous to Corollary 4.3 in \cite{MuellerStarr}. That result shows that the $\lis$ of a Mallows distributed permutation scaled by $n^{-1/2}$ can be bounded within the interval $(2 e^{-|\beta|/2}, 2 e^{|\beta|/2})$. We postpone the proof of Lemma \ref{L40} to the end of this paper. 

\begin{definition}
For any positive integer $n$ and $m \in [n]$, define
\[
Q(n, m) \coloneqq \{ (b_1, b_2, \ldots, b_m) : b_i \in [n] \text{ and } b_i < b_{i+1} \text{ for all } i\}. 
\]
\end{definition}

\begin{lemma}\label{L40}
Suppose that $\{q_n\}_{n=1}^{\infty}$ is a sequence such that $q_n > 0$ and $\lim_{n \to \infty} n(1-q_n)= \beta$ with $|\beta| < \ln{2}$. For any sequence $\{k_n\}_{n = 1}^{\infty}$ such that $k_n \in [n]$ and $\lim_{n \to \infty} k_n = \infty$, we have
\[
\lim_{n \to \infty} \max_{\bs{b} \in Q(n, k_n)} \n\left(\pi \in S_n : \frac{\lis(\pi_{\bs{b}})}{\sqrt{k_n}} \notin (2 e^{\frac{-|\beta|}{2}} - \epsilon, 2 e^{\frac{|\beta|}{2}}+\epsilon)\right) = 0
\]
for any $\epsilon > 0$.
\end{lemma}

\subsection{The scale of $\lis(\pi, \tau)$ within a rectangle}

We introduce the following way to sample a permutation according to $\m$ which will be used in the proofs.
Given $\bs{c} = (c_1, c_2, \ldots, c_m)$, where $c_i \in \mathbb{Z}^+$ and $\sum_{i=1}^{m} c_i = n$, define
\begin{align*}
\quad &d_0 \coloneqq 0, \quad d_k \coloneqq \sum_{i = 1}^{k} c_i \quad \forall k \in [m],\\
\quad &A(\bs{c}) \coloneqq \{(A_1, A_2, \ldots, A_m) : \{A_i\}_{i \in [m]} \text{ is a partition of } [n],\ |A_i| = c_i\}.
\end{align*}
Given $(A_1, \ldots, A_m) \in A(\bs{c})$, define the inversion number of $(A_1, \ldots, A_m)$ as follows,
\begin{align*}
&\inv((A_1, \ldots, A_m)) \coloneqq \\
&\qquad\quad \left|\{(x, y) : x < y \text{ and there exists $i>j$ such that }x \in A_i,\ y \in A_j \}\right|.
\end{align*}
Let $\bs{a}_i$ be the vector which consists of the numbers in $A_i$ in increasing order. There exists a bijection $f_{\bs{c}}$ between $S_n$ and
$A(\bs{c}) \times S_{c_1} \times S_{c_2} \times \cdots \times S_{c_m}$ such that, for any $\pi \in S_n$, $f_{\bs{c}}(\pi) =
((A_1, A_2, \ldots, A_m), \tau_1, \tau_2, \ldots, \tau_m)$ if and only if
\[
\{\pi(j): j \in A_i\} = \{d_{i-1} +1, d_{i-1} +2, \ldots, d_i\}, \quad \pi_{\bs{a}_i} = \tau_i, \quad \forall i \in [m].
\]
In other words, set $A_i$ consists of those indices $j$ such that $\pi(j) \in [d_{i-1} + 1, d_i]$ and $\tau_i$ denotes the relative ordering of $\{d_{i-1} +1, \ldots, d_i\}$ in $\pi$. For example, given $(A_1, A_2, A_3) = (\{1,5,6\}, \{2,4,9\}, \{3,7,8\})$ and $\tau_1 = (1, 3, 2)$, $\tau_2 = (2, 3, 1)$, $\tau_3 = (3, 2, 1)$, the corresponding permutation $\pi$ under the bijection $f_{\bs{c}}$ is $(1, 5, 9, 6, 3, 2, 8, 7, 4)$.
From the definition above, it is not hard to see that the following relation holds,
\begin{equation}
\inv(\pi) = \inv((A_1, A_2, \ldots, A_m)) + \sum_{i = 1}^m \inv(\tau_i). \label{eq:zz}
\end{equation}
Define the random variable $X_{\bs{c}}$ which takes value in $A(\bs{c})$ such that
\[
\p(X_{\bs{c}} = (A_1, A_2, \ldots, A_m))\ \propto \ q^{\inv((A_1, A_2, \ldots, A_m))}.
\]
Independent of $X_{\bs{c}}$, let $Y_1, Y_2, \ldots, Y_m$ be independent random variables such that, for any $i \in [m]$, $Y_i \sim \mu_{c_i, q}$. Define $Z \coloneqq f_{\bs{c}}^{-1}(X_{\bs{c}}, Y_1, Y_2, \ldots, Y_m)$. By (\ref{eq:zz}), we have $Z \sim \m$, since
\[
\p(Z = \pi) \ \propto \ q^{\inv(\pi)}.
\]

As our last step in preparation for the proof of Lemma \ref{L501},  we introduce the following elementary result in analysis.
\begin{lemma}\label{L38}
Suppose $\{B_i\}_{i = 1}^{\infty}$ is a partition of $\mathbb{N}$, i.e.\,$\cup_{i = 1}^{\infty}B_i = \mathbb{N}$ and $B_i \cap B_j = \varnothing, \ \forall i \neq j$. Moreover, each $B_i$ is a finite nonempty set. Given a sequence $\{x_i\}_{i = 1}^{\infty}$, if $\lim_{n \to \infty}x_{b_n} = a$, for any sequence $\{b_i\}_{i = 1}^{\infty}$ with $b_i \in B_i$, then we have $\lim_{n \to \infty}x_n = a$.
\end{lemma}

\begin{proof}
We prove the lemma by contradiction. Suppose $\lim_{n \to \infty}x_n = a$ does not hold. Then there exists $\epsilon > 0$ and  a subsequence $\{x_{n_j}\}_{j = 1}^{\infty}$ such that $x_{n_j} \notin (a - \epsilon, a + \epsilon)$ for all $j$. We can construct a sequence $\{b_i\}_{i = 1}^{\infty}$ with $b_i \in B_i$, such that $x_{b_i} \notin (a - \epsilon, a + \epsilon)$ infinitely often. Specifically, we define the sequence $\{b_i\}_{i = 1}^{\infty}$ as follows. For each $i$, if there exists an $n_j \in B_i$, let $b_i = n_j$, otherwise, let $b_i$ be the smallest number in $B_i$. Thus, we get the contradiction.
\end{proof}


For any $\pi, \tau \in S_n$, define $\bs{z}(\pi, \tau) \coloneqq \{(\frac{\pi(i)}{n}, \frac{\tau(i)}{n})\}_{i \in [n]}$. Given a rectangle $R \subset [0, 1]\times[0, 1]$, 
let $l_R(\pi, \tau)$ denote the length of the longest increasing subsequence of $\bs{z}(\pi, \tau)$ within $R$. The following lemma addresses the size of the LIS in in a small rectangle and this result will be the most crucial building block in the proof of Theorem \ref{M3}.

\begin{lemma}\label{L501}
Let $R = (x_1, x_2]\times(y_1, y_2] \subset [0, 1]\times[0, 1]$. Under the same conditions as in Lemma \ref{L100}, if $\Delta x |\beta| < \ln{2}$ , we have
\begin{equation}
\lim_{n \to \infty}\p_n\left(\frac{l_R(\pi, \tau)}{\sqrt{n\rho(R)}} \in \Big(2e^{-\Delta x |\beta| /2 } -\epsilon,\ 2 e^{\Delta x |\beta|/2} + \epsilon\Big)\right) = 1, \label{eq:L501a}
\end{equation}
for any $\epsilon > 0$, where
$\rho(R) \coloneqq \iint_R \rho(x, y)\,dxdy$ and $\Delta x \coloneqq x_2 - x_1$.
\end{lemma}

\begin{proof}
To simplify the proof, we divide the lemma into the following three cases:
\begin{itemize}[leftmargin = 20mm]
 \item[\textsf{Case} 1:] \label{case1} $\beta > 0$ or $\beta = 0$ and $q_n \le 1$ when $n$ is sufficiently large.
 \item[\textsf{Case} 2:] $\beta < 0$ or $\beta = 0$ and $q_n \ge 1$ when $n$ is sufficiently large.
 \item[\textsf{Case} 3:] $\beta = 0$.
\end{itemize}
Firstly, \textsf{Case} 3 follows from \textsf{Case} 1 and \textsf{Case} 2 because if $\lim_{n \to \infty}n(1 - q_n) = 0$, we can divide the sequence $\{q_n\}_{n = 1}^{\infty}$ into two disjoint subsequences such that one of them falls into \textsf{Case} 1 and the other falls into \textsf{Case} 2.

Next we argue that \textsf{Case} 2 follows from \textsf{Case} 1. If $\pi \sim \m$, by Proposition \ref{P3}, we have $\pi^r \sim \mu_{n, 1/q}$. Since $\lim_{n \to \infty} n(1 -q_n) = \beta \in \mathbb{R}$, we have $\lim_{n \to \infty} q_n = 1$. Hence,
\[
\lim_{n \to \infty}n (1 - 1/q_n) = \lim_{n \to \infty}n(q_n - 1)/q_n = -\beta.
\]
Therefore, \textsf{Case} 2 follows from \textsf{Case} 1 by considering the reversal of $\pi$ and $\tau$ in (\ref{eq:L501a}). Specifically, if $\pi \sim \n$ and $\tau \sim \mu_{n, q'_n}$, after reversing, we have $\pi^r \sim \mu_{n, 1/q_n}$ and $\tau^r \sim \mu_{n, 1/q'_n}$ and the $n$ points induced by $\pi$ and $\tau$ do not change, i.e., $\bs{z}(\pi, \tau) = \bs{z}(\pi^r, \tau^r)$.

To prove \textsf{Case} 1, in the following, we assume $x_1, y_1 > 0$ and $x_2, y_2 < 1$. The proofs for the cases when $x_1 = 0$ or $y_1 = 0$ or $x_2 = 1$ or $y_2 = 1$ are similar.
Let $x_3 = y_3 = 1$. Given $n \in \mathbb{N}$, we will sample $(\pi, \tau)$ according to $\p_n$ by the method introduced before Lemma \ref{L38}. Define
\begin{align*}
d_{n,i} &\coloneqq \lfloor n x_i \rfloor,& c_{n,i} &\coloneqq d_{n,i} - d_{n, i-1},&  \text{for } i &= 1, 2, 3,\\
d'_{n,i} &\coloneqq \lfloor n y_i \rfloor, & c'_{n,i} &\coloneqq d'_{n,i} - d'_{n, i-1},& \text{for } i &= 1, 2, 3,
\end{align*}
where we assume that $d_{n,0} = d'_{n,0} = 0$. Then, it is trivial that
\begin{align*}
d_{n, i} &= |\{ j \in [n] : \textstyle\frac{j}{n} \in (0, x_i]\}|, & c_{n, 2} &= |\{j \in [n] : \textstyle\frac{j}{n} \in (x_1, x_2]\}|,\\
d'_{n, i} &= |\{ j \in [n] : \textstyle\frac{j}{n} \in (0, y_i]\}|, & c'_{n, 2} &= |\{j \in [n] : \textstyle\frac{j}{n} \in (y_1, y_2]\}|.
\end{align*}
Since $\lim_{n \to \infty}\textstyle\frac{\lfloor n x \rfloor}{n} = x, \forall x \in \mathbb{R}$, it follows that
$\lim_{n \to \infty}\textstyle\frac{d_{n, i}}{n} = x_i$. Hence
\begin{equation}
\lim_{n \to \infty}\frac{c_{n, 2}}{n} = x_2 - x_1 = \Delta x. \label{eq:L501b}
\end{equation}
Next, for any nonnegative integer $i$, define $B_i \coloneqq \{n \in \mathbb{N}: c_{n, 2} = i\}$. Clearly, $\{B_i\}_{i = 0}^{\infty}$ thus defined is a partition of $\mathbb{N}$ and we show that each $B_i$ is a nonempty finite set. Since, by (\ref{eq:L501b}), $\lim_{n \to \infty}c_{n, 2} = \infty$, we conclude that each $B_i$ is a finite set. From the definition of $d_{n, i}$, it is easily seen that the sequence $\{d_{n, 1}\}$ is nondecreasing and the increment of consecutive terms is either 0 or 1. The same is true for the sequence $\{d_{n, 2}\}$. Hence, we have
\[
|c_{n+1, 2} - c_{n, 2}| = |d_{n+1, 2} - d_{n, 2} - (d_{n+1, 1}, - d_{n, 1})| \le 1.
\]
Since $c_{1, 2} \in B_0$ and $\lim_{n \to \infty}c_{n, 2} = \infty$, the inequality above guarantees that each $B_i$ is nonempty.
Next, define $\bs{c}_n \coloneqq (c_{n, 1}, c_{n, 2}, c_{n, 3})$ and $\bs{c}'_n \coloneqq (c'_{n, 1}, c'_{n, 2}, c'_{n, 3})$. Define $X_{\bs{c}_n}$ which takes values in $A(\bs{c}_n)$ such that
\[
\p(X_{\bs{c}_n} = (A_1, A_2, A_3) ) \ \propto \ q_n^{\inv((A_1, A_2, A_3))}, \qquad \forall (A_1, A_2, A_3) \in A(\bs{c}_n).
\]
Independently, define three independent random variables $Y_{n,1}, Y_{n,2}, Y_{n,3}$ such that $Y_{n,i} \sim \mu_{c_{n,i}, q_n}$.
Independent of all the variables defined above, define $X_{\bs{c}'_n}$ and $Y'_{n,1}, Y'_{n,2}, Y'_{n,3}$ in the same fashion. That is, $X_{\bs{c}'_n}$ takes value in $A(\bs{c}'_n)$ with
\[
\p(X_{\bs{c}'_n} = (A'_1, A'_2, A'_3)) \ \propto \ (q'_n)^{\inv((A'_1, A'_2, A'_3))}, \qquad \forall (A'_1, A'_2, A'_3) \in A(\bs{c}'_n)
\]
and $Y'_{n,1}, Y'_{n,2}, Y'_{n,3}$ are three independent random variables with $Y'_{n,i} \sim \mu_{c'_{n,i}, q'_n}$.
Define
\[
\pi \coloneqq f_{\bs{c}_n}^{-1}(X_{\bs{c}_n}, Y_{n, 1}, Y_{n,2}, Y_{n,3}),
\qquad \tau \coloneqq f_{\bs{c'}_n}^{-1}(X_{\bs{c}'_n}, Y'_{n,1}, Y'_{n,2}, Y'_{n,3}).
\]
From the discussion before Lemma \ref{L38}, it follows that $(\pi, \tau)$ thus defined has distribution $\p_n$. Moreover, given $X_{\bs{c}_n} = (A_1, A_2, A_3)$ and $X_{\bs{c}'_n} = (A'_1, A'_2, A'_3)$, we have
\[
\textstyle A_2 = \left\{i \in [n] : \frac{\pi(i)}{n} \in (x_1, x_2]\right\}, \quad A'_2 = \left\{i \in [n] : \frac{\tau(i)}{n} \in (y_1, y_2]\right\}.
\]
Hence, we have
\begin{equation}
\textstyle A_2 \cap A'_2 = \left\{ i \in [n] : \left(\frac{\pi(i)}{n}, \frac{\tau(i)}{n}\right) \in R \right\}. \label{eq:L501bb}
\end{equation}
Define $M = |\bs{z}(\pi, \tau) \cap R|$, i.\,e.\,$M$ denotes the number of points $\{(\frac{\pi(i)}{n}, \frac{\tau(i)}{n})\}_{i = 1}^n$ within $R$. Then, by (\ref{eq:L501bb}), we have $M = |A_2 \cap A'_2|$. Hence, $M$ only depends on the values of $X_{\bs{c}_n}$ and $X_{\bs{c}'_n}$ and is independent of $\cup_{i\in[3]}\{Y_{n, i}, Y'_{n,i}\}$. Next note that, conditioning on $X_{\bs{c}_n} = (A_1, A_2, A_3)$ and $X_{\bs{c}'_n} = (A'_1, A'_2, A'_3)$, $l_R(\pi, \tau)$ is determined by $Y_{n,2}$ and $Y'_{n,2}$. To see this, we first define a new function $I$ as follows, given any finite set $A \subset \mathbb{Z}$ and any $a \in A$, define $I(A, a) \coloneqq k$ if $a$ is the $k$-th smallest number in $A$. Suppose $A_2 \cap A'_2 = \{a_j\}_{j \in [M]}$ with $a_1 < a_2 < \cdots < a_M$. Define $\bs{b}\in Q(c_{n, 2}, M)$ and $\bs{b}' \in Q(c'_{n, 2}, M)$ by
\begin{equation}\label{eq:L501bc}
\begin{aligned}
\bs{b} &\coloneqq (I(A_2, a_1), I(A_2, a_2), \ldots, I(A_2, a_M)),\\
\bs{b}' &\coloneqq (I(A'_2, a_1), I(A'_2, a_2), \ldots, I(A'_2, a_M)).
\end{aligned}
\end{equation}
Note that $\bs{b}$ and $\bs{b}'$ are determined by $A_2$ and $A'_2$. Then, we have
\begin{equation}
l_R(\pi, \tau) = \lis((Y_{n, 2})_{\bs{b}}, (Y'_{n, 2})_{\bs{b}'}). \label{eq:L501c}
\end{equation}
Indeed, conditioning on $X_{\bs{c}_n} = (A_1, A_2, A_3)$, we know that $\{\pi(i) : i \in A_2\} = \{d_{n,1}+1, d_{n,1}+2, \ldots, d_{n,2}\}$. And the value of $Y_{n, 2}$ determines the relative ordering of $\pi(i)$ for those $i \in A_2$. Similarly, the value of $Y'_{n, 2}$ determines the relative ordering of $\tau(i)$ for those $i \in A'_2$.

Now we are in the position to prove (\ref{eq:L501a}) for \hyperref[case1]{\textsf{Case} 1}. From the discussion above and Lemma \ref{L38}, it suffices to show that, for any sequence $\{s_n\}_{n=1}^{\infty}$ with $s_n \in B_n$, i.e., when $c_{s_n, 2} = n$, we have
\begin{equation}
\lim_{n \to \infty}\textstyle\p_{s_n}\left(\frac{l_R(\pi, \tau)}{\sqrt{s_n\rho(R)}} \in \Big(2e^{-\Delta x \beta /2 } -\epsilon,\ 2 e^{\Delta x \beta/2} + \epsilon\Big)\right) = 1, \label{eq:L501d}
\end{equation}
for any $\epsilon > 0$. Note that by the definition of $\p_{s_n}$ in Lemma \ref{L100},  $\pi$ and $\tau$ above are of size $s_n$ with $\pi \sim \mu_{s_n, q_{s_n}}$, $\tau \sim \mu_{s_n, q'_{s_n}}$.

We separate the proof of (\ref{eq:L501d}) into two parts. Specifically, we need to show that, for any $\epsilon > 0$,
\begin{equation}
\lim_{n \to \infty}\textstyle\p_{s_n}\left(\frac{l_R(\pi, \tau)}{\sqrt{s_n\rho(R)}} < 2 e^{\Delta x \beta/2} + \epsilon\right) = 1, \label{eq:L501e1}
\end{equation}
and
\begin{equation}
\lim_{n \to \infty}\textstyle\p_{s_n}\left(\frac{l_R(\pi, \tau)}{\sqrt{s_n\rho(R)}} > 2 e^{- \Delta x \beta/2} - \epsilon\right) = 1. \label{eq:L501e2}
\end{equation}

Since $\{s_n\}_{n \ge 1}$ is a subsequence of $\{i\}_{i \ge 0}$, $\lim_{n \to \infty} s_n = \infty$. Hence, by (\ref{eq:L501b}) and the fact that $c_{s_n, 2} = n$, we get
\[
\lim_{n \to \infty}\frac{n}{s_n} = \lim_{n \to \infty}\frac{c_{s_n, 2}}{s_n} = \Delta x.
\]
Thus,
\begin{equation}
\lim_{n \to \infty}n(1 - q_{s_n}) = \lim_{n \to \infty}\frac{n}{s_n}s_n(1 - q_{s_n}) = \Delta x \beta < \ln{2}. \label{eq:L501f}
\end{equation}
To prove (\ref{eq:L501e1}), for any $\epsilon > 0$, we can choose $\epsilon_1>0$ sufficiently small such that
\begin{equation}
(1 - \epsilon_1)(2 e^{\Delta x \beta/2} + \epsilon) > 2 e^{\Delta x \beta/2}. \label{eq:L501g}
\end{equation}
For this fixed $\epsilon_1$, we can choose $\delta > 0$ such that
\begin{equation}
\textstyle\sqrt{\frac{\rho(R)}{\rho(R) + \delta}} > 1 - \epsilon_1. \label{eq:L501h}
\end{equation}
Given $n \in \mathbb{N}$, define $k_n = \lfloor s_n(\rho(R) + \delta) \rfloor$. Clearly, we have $\lim_{n \to \infty}k_n = \infty$. Hence, by Lemma \ref{L40}, (\ref{eq:L501f}) and (\ref{eq:L501g}), there exists $N_1 > 0$ such that, for any $n > N_1$, we have
\begin{equation}
\min_{\bs{b}\in Q(n, k_n)}\textstyle\mu_{n, q_{s_n}}\left(\eta \in S_n : \frac{\lis(\eta_{\bs{b}})}{\sqrt{k_n}} < (1 - \epsilon_1)\big(2 e^{\Delta x \beta/2} + \epsilon\big)\right) > 1 - \epsilon.          \label{eq:L501i}
\end{equation}
Given $\bs{b}\in Q(n, k_n)$, for any $\bs{b}'$ which is a subsequence of $\bs{b}$, we have $\lis(\eta_{\bs{b}}) \ge \lis(\eta_{\bs{b}'})$. Thus we can make (\ref{eq:L501i}) stronger as follows,
\begin{equation}
\min_{\bs{b}\in \bar{Q}(n, k_n)}\textstyle\mu_{n, q_{s_n}}\left(\eta \in S_n : \frac{\lis(\eta_{\bs{b}})}{\sqrt{k_n}} < (1 - \epsilon_1)\big(2 e^{\Delta x \beta/2} + \epsilon\big)\right) > 1 - \epsilon,         \label{eq:L501j}
\end{equation}
where $\bar{Q}(n, k_n) = \cup_{i \in [k_n]} Q(n, i)$.
Since $\lim_{n \to \infty} s_n = \infty$, we have
\begin{equation}
\lim_{n \to \infty}s_n(1 - q_{s_n}) = \beta \quad \text{ and } \quad \lim_{n \to \infty}s_n(1 - q'_{s_n}) = \gamma. \label{eq:L501jj}
\end{equation}
Hence, by Lemma \ref{L100}, there exists $N_2 >0$ such that, for any $n > N_2$, we have
\begin{equation}
\textstyle\p_{s_n}\left(\frac{|\bs{z}(\pi, \tau) \cap R|}{s_n} \le \rho(R) + \delta\right) > 1 - \epsilon.  \label{eq:L501k}
\end{equation}
In the following, let $E_n(A_2, A'_2)$ denote the event that the second entries of $X_{\bs{c}_{s_n}}$ and $X_{\bs{c}'_{s_n}}$ are $A_2$ and $A'_2$ respectively. Let $\p$ denote the probability space on which $(X_{\bs{c}_{s_n}}, Y_{s_n, 1}, Y_{s_n, 2}, Y_{s_n, 3})$ and $(X_{\bs{c}'_{s_n}}, Y'_{s_n, 1}, Y'_{s_n, 2}, Y'_{s_n, 3})$ are defined. Then, for any $n > \max(N_1, N_2)$, we have
\begin{align*}
  &\textstyle\p_{s_n}\Big(\frac{l_R(\pi, \tau)}{\sqrt{s_n\rho(R)}} < 2 e^{\Delta x \beta/2} + \epsilon\Big)\\
  \ge&\sum_{|A_2\cap A'_2| \le k_n}\textstyle\p\Big(\frac{l_R(\pi, \tau)}{\sqrt{s_n\rho(R)}} < 2 e^{\Delta x \beta/2} + \epsilon \ \big| \
  E_n(A_2, A'_2)\Big)\times\p(E_n(A_2, A'_2))\\
  =&\sum_{|A_2\cap A'_2| \le k_n}\textstyle\p\Big(\frac{\lis((Y_{s_n, 2})_{\bs{b}},\ (Y'_{s_n, 2})_{\bs{b}'})}{\sqrt{s_n\rho(R)}} < 2 e^{\Delta x \beta/2} + \epsilon \Big)\times\p(E_n(A_2, A'_2))\\
  \ge&\sum_{|A_2\cap A'_2| \le k_n}\textstyle\mu_{n, q_{s_n}}\Big(\frac{\lis(\eta_{\bs{b}})}{\sqrt{s_n\rho(R)}} < 2 e^{\Delta x \beta/2} + \epsilon \Big)\times\p(E_n(A_2, A'_2))\\
  =&\sum_{|A_2\cap A'_2| \le k_n}\textstyle\mu_{n, q_{s_n}}\Big(\frac{\lis(\eta_{\bs{b}})}{\sqrt{s_n(\rho(R)+ \delta)}} < \frac{\sqrt{\rho(R)}}{\sqrt{\rho(R)+\delta}}(2 e^{\Delta x \beta/2} + \epsilon) \Big)\\
  & \qquad\qquad\qquad\qquad\qquad\qquad\qquad\qquad\qquad \times\p(E_n(A_2, A'_2))\\
  \ge&\sum_{|A_2\cap A'_2| \le k_n}\textstyle\mu_{n, q_{s_n}}\Big(\frac{\lis(\eta_{\bs{b}})}{\sqrt{k_n}} < (1 - \epsilon_1)(2 e^{\Delta x \beta/2} + \epsilon) \Big)\times\p(E_n(A_2, A'_2))\\
  \ge& \ (1 - \epsilon)\times\sum_{|A_2\cap A'_2| \le k_n}\p(E_n(A_2, A'_2))\\
  =& \ (1 - \epsilon) \times \p_{s_n}\big( |\bs{z}(\pi, \tau) \cap R| \le k_n\big)\\
  =& \ (1 - \epsilon) \times \p_{s_n}\big( |\bs{z}(\pi, \tau) \cap R| \le s_n(\rho(R)+\delta)\big)\\
  >& \ (1 - \epsilon)^2.
\end{align*}
The first equality follows by (\ref{eq:L501c}) and the independence of $(X_{\bs{c}_{s_n}}, X_{\bs{c}'_{s_n}})$ and $(Y_{s_n, 2}, Y'_{s_n, 2})$. Note that $\bs{b}$ and $\bs{b}'$ are determined by $A_2$ and $A'_2$ as in (\ref{eq:L501bc}). The second inequality follows by Lemma \ref{L360}, since $Y_{s_n, 2}$ and $Y'_{s_n, 2}$ are independent with $Y_{s_n, 2} \sim \mu_{n, q_{s_n}}$. The third inequality follows by (\ref{eq:L501h}) and the fact that $k_n = \lfloor s_n(\rho(R) + \delta) \rfloor \le s_n(\rho(R) + \delta)$. The fourth inequality follows by (\ref{eq:L501j}) and the fact that the dimension of $\bs{b}$ equals to $|A_2\cap A'_2|$. The last inequality follows by (\ref{eq:L501k}). Hence, (\ref{eq:L501e1}) follows.

The proof of (\ref{eq:L501e2}) follows in a similar way as the proof of (\ref{eq:L501e1}). First, by (\ref{eq:L501f}) and the fact that $\lim_{n \to \infty}q_n = 1$, we have
\begin{equation}
\lim_{n \to \infty}n(1 - 1/q_{s_n}) = \lim_{n \to \infty}\frac{n(q_{s_n} -1)}{q_{s_n}} = - \Delta x \beta > - \ln{2}. \label{eq:L501l}
\end{equation}
For any $\epsilon >0$, we can choose $\epsilon_1 > 0$ sufficiently small such that
\begin{equation}
(1 + \epsilon_1)(2 e^{- \Delta x \beta/2} - \epsilon) < 2 e^{ -\Delta x \beta/2}.  \label{eq:L501m}
\end{equation}
For this fixed $\epsilon_1$, we can choose $\delta > 0$ such that
\begin{equation}
\textstyle\sqrt{\frac{\rho(R)}{\rho(R) - \delta}} < 1 + \epsilon_1.  \label{eq:L501n}
\end{equation}
Given $n \in \mathbb{N}$, define $k'_n = \lceil s_n(\rho(R) - \delta) \rceil$. Clearly, we have $\lim_{n \to \infty}k'_n = \infty$. Moreover, under conditions of \hyperref[case1]{\textsf{Case} 1}, $1/q_n \ge 1$ for sufficiently large $n$. Hence, by Lemma \ref{L40}, (\ref{eq:L501l}) and (\ref{eq:L501m}), there exist $N_3 > 0$ such that, for any $n > N_3$, we have
\begin{equation}
\min_{\bs{b}\in Q(n, k'_n)}\textstyle\mu_{n, 1/q_{s_n}}\left(\eta \in S_n : \frac{\lis(\eta_{\bs{b}})}{\sqrt{k'_n}} > (1 + \epsilon_1)\big(2 e^{- \Delta x \beta/2} - \epsilon\big)\right) > 1 - \epsilon.      \label{eq:L501o}
\end{equation}
Given $\bs{b}\in Q(n, k'_n)$, for any $\bs{b}'$ such that $\bs{b}$ is a subsequence of $\bs{b}'$, we have $\lis(\eta_{\bs{b}}) \le \lis(\eta_{\bs{b}'})$. Thus we can make (\ref{eq:L501o}) stronger as follows,
\begin{equation}
\min_{\bs{b}\in \hat{Q}(n, k'_n)}\textstyle\mu_{n, 1/q_{s_n}}\left(\eta \in S_n : \frac{\lis(\eta_{\bs{b}})}{\sqrt{k'_n}} > (1 + \epsilon_1)\big(2 e^{- \Delta x \beta/2} - \epsilon\big)\right) > 1 - \epsilon,          \label{eq:L501p}
\end{equation}
where $\hat{Q}(n, k'_n) = \cup_{k'_n\le i \le n}Q(n, i)$.
By (\ref{eq:L501jj}) and Lemma \ref{L100}, there exists $N_4 >0$ such that, for any $n > N_4$, we have
\begin{equation}
\textstyle\p_{s_n}\left(\frac{|\bs{z}(\pi, \tau) \cap R|}{s_n} \ge \rho(R) - \delta\right) > 1 - \epsilon.  \label{eq:L501q}
\end{equation}
Then, assuming the notations defined in the proof of (\ref{eq:L501e1}), for any $n > \max(N_3, N_4)$, we have
\begin{align*}
  &\p_{s_n}\textstyle\Big(\frac{l_R(\pi, \tau)}{\sqrt{s_n\rho(R)}} > 2 e^{-\Delta x \beta/2} - \epsilon\Big)\\
  \ge&\sum_{|A_2\cap A'_2| \ge k_n}\textstyle\p\Big(\frac{l_R(\pi, \tau)}{\sqrt{s_n\rho(R)}} > 2 e^{-\Delta x \beta/2} - \epsilon \ \big| \
  E_n(A_2, A'_2)\Big)\times\p(E_n(A_2, A'_2))\\
  =&\sum_{|A_2\cap A'_2| \ge k'_n}\textstyle\p\Big(\frac{\lis((Y_{s_n, 2})_{\bs{b}},\ (Y'_{s_n, 2})_{\bs{b}'})}{\sqrt{s_n\rho(R)}}  > 2 e^{-\Delta x \beta/2} - \epsilon \Big)\times\p(E_n(A_2, A'_2))\\
  \ge&\sum_{|A_2\cap A'_2| \ge k'_n}\textstyle\mu_{n, 1/q_{s_n}}\Big(\frac{\lis(\eta_{\bar{\bs{b}}})}{\sqrt{s_n\rho(R)}} > 2 e^{-\Delta x \beta/2} - \epsilon \Big)\times\p(E_n(A_2, A'_2))\\
  =&\sum_{|A_2\cap A'_2| \ge k'_n}\textstyle\mu_{n, 1/q_{s_n}}\Big(\frac{\lis(\eta_{\bar{\bs{b}}})}{\sqrt{s_n(\rho(R) - \delta)}} > \frac{\sqrt{\rho(R)}}{\sqrt{\rho(R)-\delta}}(2 e^{-\Delta x \beta/2} - \epsilon) \Big)\\
  & \qquad\qquad\qquad\qquad\qquad\qquad\qquad\qquad\qquad\qquad\qquad \times\p(E_n(A_2, A'_2))\\
  \ge&\sum_{|A_2\cap A'_2| \ge k'_n}\textstyle\mu_{n, 1/q_{s_n}}\Big(\frac{\lis(\eta_{\bar{\bs{b}}})}{\sqrt{k'_n}} > (1 + \epsilon_1)(2 e^{- \Delta x \beta/2} - \epsilon) \Big)\times\p(E_n(A_2, A'_2))\\
  \ge& \ (1 - \epsilon)\times\sum_{|A_2\cap A'_2| \ge k'_n}\p(E_n(A_2, A'_2))\\
  =& \ (1 - \epsilon) \times \p_{s_n}\big( |\bs{z}(\pi, \tau) \cap R| \ge k'_n\big)\\
  =& \ (1 - \epsilon) \times \p_{s_n}\big( |\bs{z}(\pi, \tau) \cap R| \ge s_n(\rho(R)-\delta)\big)\\
  >& \ (1 - \epsilon)^2.
\end{align*}
The first equality follows by (\ref{eq:L501c}) and the independence of $(X_{\bs{c}_{s_n}}, X_{\bs{c}'_{s_n}})$ and $(Y_{s_n, 2}, Y'_{s_n, 2})$. The second inequality follows by Lemma \ref{L370}, since $Y_{s_n, 2}$ and $Y'_{s_n, 2}$ are independent with $Y_{s_n, 2} \sim \mu_{n, q_{s_n}}$. The third inequality follows by (\ref{eq:L501n}) and the fact that 
$
k'_n = \lceil s_n(\rho(R) - \delta) \rceil \ge s_n(\rho(R) - \delta).
$
The fourth inequality follows by (\ref{eq:L501p}) and the fact that $\bar{\bs{b}}$ has the same dimension as of $\bs{b}$ which equals to $|A_2\cap A'_2|$. The last inequality follows by (\ref{eq:L501q}). Hence, (\ref{eq:L501e2}) follows and this completes the proof of Lemma \ref{L501}.
\end{proof}

\subsection{Deuschel and Zeitouni's approach}

The following lemma establishes certain degree of smoothness of the densities $u$ and $\rho$ defined in Lemma \ref{L100}.


\begin{lemma}\label{L101}
The density functions $u(x, y, \beta)$ defined in (\ref{eq:M1}) and $\rho(x, y)$ defined in (\ref{eq:L100z}) satisfy the following,
\begin{itemize}
\item[(a)] $e^{-|\beta|} \le u(x, y, \beta) \le e^{|\beta|}$, $e^{-|\beta|-|\gamma|} \le \rho(x, y) \le e^{|\beta|+|\gamma|}$,
\item[(b)] $u(x, y, \beta)\in C_b^1,\, \rho(x, y) \in C_b^1$,
\item[(c)] $\max{\left(\big|\frac{\partial u}{\partial x}\big|, \big|\frac{\partial u}{\partial y}\big|\right)} \le |\beta|e^{|\beta|}$,
\item[(d)] $\max{\left(\big|\frac{\partial \rho}{\partial x}\big|, \big|\frac{\partial \rho}{\partial y}\big|\right)} \le
(|\beta|+|\gamma|)e^{|\beta|+|\gamma|}$,
\end{itemize}
where $(x, y) \in [0, 1] \times [0, 1]$.
\end{lemma}

\begin{proof}
First we show that $e^{-|\beta|} \le u(x, y, \beta) \le e^{|\beta|}$ for any $0 \le x, y \le 1$. Here we assume $\beta > 0$. The proof for the case when $\beta < 0$ is similar. By (\ref{eq:M1}), we have
\begin{align}
u(x, y, \beta) &= \textstyle \frac{(\beta/2) \sinh(\beta/2)}{\left(e^{\beta/4} \cosh(\beta[x-y]/2)-e^{-\beta/4}\cosh(\beta[x+y-1]/2)\right)^2} \nonumber\\
 &=\textstyle \frac{\beta (e^{\beta} - 1)}{\left(2e^{\beta/2} \cosh(\beta[x-y]/2)-2\cosh(\beta[x+y-1]/2)\right)^2}. \label{eq:L101x}
\end{align}
Since $-1 \le x - y \le 1$ and $-1 \le x + y -1 \le 1$, we have
\begin{align}
2e^{\beta/2} &\le 2e^{\beta/2} \cosh(\beta[x-y]/2) \le e^{\beta} + 1, \label{eq:L101s}\\
2 &\le 2\cosh(\beta[x+y-1]/2) \le e^{\beta/2} + e^{-\beta/2}. \label{eq:L101t}
\end{align}
Since $e^{\beta/2} + e^{-\beta/2} < 2e^{\beta/2}$, from (\ref{eq:L101s}) and (\ref{eq:L101t}), we have
\begin{equation}\label{eq:L101y}
e^{\beta/2} - e^{-\beta/2} \le 2e^{\beta/2} \cosh(\beta[x-y]/2)-2\cosh(\beta[x+y-1]/2) \le e^{\beta}-1. 
\end{equation}
By (\ref{eq:L101x}) and (\ref{eq:L101y}), it follows that
\begin{equation}\label{eq:L101w}
\textstyle \frac{\beta}{e^{\beta}-1} \le u(x, y, \beta) \le 
\frac{\beta(e^{\beta}-1)}{(e^{\beta/2} - e^{-\beta/2})^2}.
\end{equation}
It is easily verified that
\begin{align}
\textstyle \frac{\beta}{e^{\beta}-1} \ge e^{-\beta} &\Longleftrightarrow 
e^{-\beta} \ge 1 - \beta, \label{eq:L101v}\\
\textstyle \frac{\beta(e^{\beta}-1)}{(e^{\beta/2} - e^{-\beta/2})^2} \le e^{\beta} &\Longleftrightarrow (e^{\beta}-1)(e^{\beta} - 1 - \beta) \ge 0. \label{eq:L101u}
\end{align}
By the inequality $e^x \ge 1 + x$, the right-hand side of (\ref{eq:L101v}) and (\ref{eq:L101u}) hold. It follows from (\ref{eq:L101w}) and the left-hand side of (\ref{eq:L101v}) and (\ref{eq:L101u}) that
\[
e^{-\beta} \le u(x, y, \beta) \le e^{\beta}, \qquad \forall \ 0 \le x, y \le 1.
\]
By the definition of $\rho(x, y)$, it follows trivially that 
\[
e^{-|\beta| - |\gamma|} \le \rho(x, y) \le e^{|\beta| + |\gamma|}, \qquad \forall \ 0 \le x, y \le 1.
\]
In \cite{Starr}, Starr shows that $\frac{\partial^2 \ln{u(x, y, \beta)}}{\partial x \partial y}  = 2 \beta u(x, y, \beta)$. Thus
\begin{equation}\label{eq:L101z}
\textstyle \int_{0}^{x} u(t, y, \beta) \, dt =
\frac{1}{2\beta} \left( \frac{\partial \ln{u(x, y, \beta)}}{\partial y} - \frac{\partial \ln{u(0, y, \beta)}}{\partial y} \right).
\end{equation}
By direct calculation, we have $u(1, y, \beta) = \frac{\beta e^{\beta y}}{e^{\beta} - 1}$, $u(0, y, \beta) = \textstyle \frac{\beta e^{-\beta y}}{1 - e^{-\beta}}$. Therefore, we get $\frac{\partial \ln{u(1, y, \beta)}}{\partial y} = \beta$ and $\frac{\partial \ln{u(0, y, \beta)}}{\partial y} = - \beta$.
By (\ref{eq:L101z}), it follows that
\begin{equation}
\textstyle \frac{\partial u(x, y, \beta)}{\partial y} = 2 \beta u(x, y, \beta) \left( \int_{0}^{x} u(t, y, \beta) \, dt - \frac{1}{2} \right),
\label{eq:L101a}
\end{equation}
and
\begin{equation}\label{eq:L101aa}
\textstyle \int_{0}^{x} u(t, y, \beta) \, dt \le \int_{0}^{1} u(t, y, \beta) \, dt = 1.
\end{equation}
From (\ref{eq:L101a}) and (\ref{eq:L101aa}), we get
\begin{equation}
\textstyle \left|\frac{\partial u}{\partial y}\right| \le |\beta|u(x, y, \beta) \le |\beta|e^{|\beta|}.
\end{equation}
Since $u(x, y, \beta)$ is uniformly continuous on $[0, 1]\times[0, 1]$, $\int_{0}^{x} u(t, y, \beta) \, dt$ is also continuous on $[0, 1]\times[0, 1]$. Hence, by (\ref{eq:L101a}), $\frac{\partial u}{\partial y}$ is bounded and continuous on $[0, 1]\times[0, 1]$. Similar argument can be made for $\frac{\partial u}{\partial x}$. Thus we have shown that $u(x, y, \beta) \in C_b^1$ and
\[
\textstyle \max\left(\big|\frac{\partial u}{\partial x}\big|,  \big|\frac{\partial u}{\partial y}\big|\right) \le |\beta|e^{|\beta|}.
\]
Next, since $\big|\frac{\partial u(x, t, \beta)}{\partial x}\cdot u(t, y, \gamma)\big| \le |\beta|e^{|\beta|+|\gamma|}$ for any $0 \le x, y, t \le 1$, by dominated convergence theorem, we have
\begin{equation}\label{eq:L101b}
\textstyle \frac{\partial \rho(x, y)}{\partial x} = \frac{\partial }{\partial x}\Big(\int_{0}^{1}u(x, t, \beta)u(t, y, \gamma)\, dt\Big) = \int_{0}^{1}\frac{\partial u(x, t, \beta)}{\partial x}u(t, y, \gamma)\,dt.
\end{equation}
Hence, $\big|\frac{\partial \rho}{\partial x}\big| \le |\beta|e^{|\beta|+|\gamma|}$. Moreover, $\frac{\partial u(x, t, \beta)}{\partial x}\cdot u(t, y, \gamma)$ as a function of $x, y, t$ is uniformly continuous on $[0, 1]\times[0, 1]\times[0, 1]$. Thus, by (\ref{eq:L101b}), $\frac{\partial \rho}{\partial x}$ is continuous on $[0, 1]\times[0, 1]$. By a similar argument, it can be shown that $\frac{\partial \rho}{\partial y}$ is continuous on $[0, 1]\times[0, 1]$, and $\big|\frac{\partial \rho}{\partial y}\big| \le |\gamma|e^{|\beta|+|\gamma|}$. Therefore, $\rho(x, y) \in C_b^1$ and
\[
\textstyle \max{\left(\big|\frac{\partial \rho}{\partial x}\big|, \big|\frac{\partial \rho}{\partial y}\big|\right)} \le (|\beta|+|\gamma|)e^{|\beta|+|\gamma|}. \qedhere
\]
\end{proof}


The next lemma shows that for any non-decreasing curve in the unit square, in a strip of small width around it, with probability going to 1, there exists an increasing subsequence whose length can be bounded from below.
The proof of Lemma \ref{L54} uses similar arguments as in the proof of Lemma 8 in \cite{DZ95}. Before stating the lemma, we need the following notation.

\begin{definition}\label{D51}
Let $B_{\nearrow}$ be the set of nondecreasing, right continuous functions $\phi : [0, 1] \rightarrow [0, 1]$. For $\phi \in B_{\nearrow}$, we have $\phi(x) = \int_{0}^{x}\dot{\phi}(t)\,dt + \phi_s(x)$, where $\phi_s$ is singular and has a zero derivative almost everywhere. Let $\rho(x, y)$ be the density defined in (\ref{eq:L100z}). Define function $J : B_{\nearrow} \rightarrow \mathbb{R}$,
\[
J(\phi) \coloneqq \int_0^1 \sqrt{\dot{\phi}(x) \rho(x, \phi(x))} \, dx \quad  \text{ and } \quad \bar{J} \coloneqq \sup_{\phi \in B_{\nearrow}} J(\phi).
\]
\end{definition}
\begin{remark*}
By Theorems 3 and 4 in \cite{DZ95} it follows from Lemma \ref{L101} (a) and (b), that
\[
\sup_{\phi \in B_{\nearrow}} J(\phi) = \sup_{\phi \in B^1_{\nearrow}} J(\phi),
\]
where $B^1_{\nearrow}$ is defined in Theorem \ref{M3}. Hence we use the same notation $\bar{J}$ to denote the supremum over $B_{\nearrow}$.
\end{remark*}

Given a function $\phi(x)$ and any $\delta > 0$, we say that a point $(x, y)$ is in the $\delta$ neighborhood of $\phi$ if $\phi(x) - \delta < y < \phi(x) + \delta$.

\begin{lemma}\label{L54}
Under the same conditions as in Theorem \ref{M3}, for any $\phi \in B^1_{\nearrow}$ and any $\delta, \epsilon > 0$, define the event
\begin{align*}
E_n \coloneqq \Big\{&(\pi, \tau) \in S_n \times S_n :
\exists \text{ an increasing subsequence of } \textstyle\big\{\big(\frac{\pi(i)}{n}, \frac{\tau(i)}{n}\big)\big\}_{i \in [n]}\\
 &\text{ which is wholly contained in the $\delta$ neighborhood of } \phi(\cdot)\\
 &\text{ and the length of which is greater than } 2J(\phi)(1 - \epsilon)\sqrt{n}\ \Big\}.
\end{align*}
Then
\[
\lim_{n \to \infty} \p_n ( E_n ) = 1.
\]
\end{lemma}

\begin{proof}
Given $\delta, \epsilon > 0$, fix an integer $K$. Let $\Delta x \coloneqq 1/K$. Let $x_i \coloneqq i \Delta x$ and $y_i \coloneqq \phi(x_i)$ for $i \in [K]$. Let $x_0 \coloneqq 0$, $y_0 \coloneqq 0$. Define the rectangles $R_i \coloneqq [x_{i-1}, x_{i}] \times [y_{i-1}, y_{i}]$ for $i \in [K]$. Since $\phi$ is in $C^1_b$, for any $0 < \delta' <1$, we can choose $K$ large enough such that
\begin{align}
&\max_i(y_{i} - y_{i-1}) < \delta, \quad e^{-\Delta x |\beta|/2} > 1 - \delta',\quad \Delta x |\beta|<\ln{2}\label{eq:L54a} \\
&\max_{i} \max_{x, y \in R_i} \max\left(\frac{\rho(x, y)}{\rho(x_i, y_i)}, \frac{\rho(x_i, y_i)}{\rho(x, y)} \right) < \frac{1}{1 - \delta'}, \label{eq:L54b}
\end{align}
and
\begin{equation}
 \sum_{i = 1}^{K} \sqrt{\rho(x_i, y_i)(y_i - y_{i -1}) \Delta x }  > (1 - \delta')J(\phi).  \label{eq:L54c}
\end{equation}
(\ref{eq:L54b}) follows from the uniform continuity of $\rho(x, y)$ on $[0, 1] \times [0, 1]$ and the fact that $\rho(x, y)$ is bounded away from 0, which is proved in Lemma \ref{L101} (a). (\ref{eq:L54c}) follows since
\begin{align*}
 &\lim_{K \to \infty}  \sum_{i = 1}^{K} \sqrt{\rho(x_i, y_i)(y_i - y_{i -1}) \Delta x } \\
 = &\lim_{K \to \infty}  \sum_{i = 1}^{K} \sqrt{\rho(x_i, y_i)\frac{y_i - y_{i -1}}{x_i - x_{i-1}}} \, \Delta x \\
 = &\, J(\phi),
\end{align*}
where the last equality follows from the definition of Riemann integral, the mean value theorem and the fact that $\phi \in C_b^1$.
Next, for any $i \in [K]$, define $\rho(R_i) \coloneqq \iint_{R_i}\rho(x, y)\,dxdy$. By (\ref{eq:L54b}), we have
\[
\frac{\rho(R_i)}{1 - \delta'} > \rho(x_i, y_i)(y_i - y_{i-1})\Delta x.
\]
Hence, for any $i \in [K]$, we have
\begin{equation}
\frac{l_{R_i}(\pi, \tau)}{2 \sqrt{n \rho(x_i, y_i)(y_i - y_{i -1}) \Delta x }} \ge
\frac{l_{R_i}(\pi, \tau)\sqrt{1 - \delta'}}{2 \sqrt{n \rho(R_i)}}.  \label{eq:L54d}
\end{equation}
By fixing the $\epsilon$ in Lemma \ref{L501} to be $2 \delta'$, we have
\begin{equation}
\lim_{n \to \infty}\p_n\left(\frac{l_{R_i}(\pi, \tau)}{\sqrt{n\rho(R_i)}} > 2e^{-\Delta x |\beta| /2 } - 2 \delta'\right) = 1. \label{eq:L54e}
\end{equation}
Moreover,
\begin{align}
 &\p_n\left(\frac{l_{R_i}(\pi, \tau)}{2 \sqrt{n \rho(x_i, y_i)(y_i - y_{i -1}) \Delta x }} > (1 - 2 \delta')\sqrt{1 - \delta'}\right)\label{eq:L54f}\\
 \ge \ &\p_n\left(\frac{l_{R_i}(\pi, \tau)}{2 \sqrt{n \rho(R_i)}} > 1 - 2 \delta'\right)\nonumber\\
 \ge \ &\p_n\left(\frac{l_{R_i}(\pi, \tau)}{\sqrt{n\rho(R_i)}} > 2e^{-\Delta x |\beta| /2 } - 2 \delta'\right).  \nonumber
\end{align}
The first inequality follows by (\ref{eq:L54d}), and the second inequality follows by (\ref{eq:L54a}), since
\[
2e^{-\Delta x |\beta| /2 } - 2 \delta' > 2(1 - \delta') - 2\delta' = 2(1 - 2 \delta').
\]
Hence, by (\ref{eq:L54e}) and (\ref{eq:L54f}), we get
\begin{equation}
\lim_{n \to \infty} \p_n\left(\frac{l_{R_i}(\pi, \tau)}{2 \sqrt{n \rho(x_i, y_i)(y_i - y_{i -1}) \Delta x }} > (1 - 2 \delta')\sqrt{1 - \delta'}\right) = 1,    \label{eq:L54g}
\end{equation}
for any $i \in [K]$.
Note that by concatenating the increasing subsequences of $\big\{\big(\frac{\pi(i)}{n}, \frac{\tau(i)}{n}\big)\big\}_{i \in [n]}$ in each $R_i$ we get a increasing subsequence in $[0, 1]\times[0, 1]$ which is wholly contained in a $\delta$ neighborhood of $\phi$. Combining (\ref{eq:L54c}) and (\ref{eq:L54g}), it follows that, with probability converging to 1 as $n \to \infty$, there exists an increasing subsequence of $\big\{\big(\frac{\pi(i)}{n}, \frac{\tau(i)}{n}\big)\big\}_{i \in [n]}$ in a $\delta$~neighborhood of $\phi$ whose length is at least
\[
\sum_{i = 1}^{K} 2 \sqrt{n} (1 - 2 \delta')\sqrt{1 - \delta'} \sqrt{\rho(x_i, y_i)(y_i - y_{i -1}) \Delta x} > 2 \sqrt{n} (1 - 2 \delta')(1 - \delta')^{\frac{3}{2}} J(\phi).
\]
The lemma follows since we can choose $\delta'$ small enough in the first place such that $(1 - 2 \delta')(1 - \delta')^{\frac{3}{2}} > 1 - \epsilon$.

\end{proof}

\begin{definition}\label{D520}
Given $K, L \in \mathbb{N}$, define
\[
\bs{B}_{KL} \coloneqq \{(b_0, b_1, \ldots, b_K) \in \mathbb{Z}^{K+1} : 0 = b_0 \le b_1 \le \cdots \le b_K = KL-1 \}.
\]
\end{definition}

\begin{definition}\label{D52}
Given $K, L \in \mathbb{N}$ and $\bs{b} = (b_0, b_1, \ldots, b_K) \in \bs{B}_{KL}$, for any $i \in [K]$, define the rectangle $R_i \coloneqq ((i-1)\Delta x, i \Delta x] \times (b_{i-1}\Delta y, (b_i + 1) \Delta y]$, where $\Delta x \coloneqq \frac{1}{K}$ and $\Delta y \coloneqq \frac{1}{KL}$. Let $M_i \coloneqq \sup_{(x, y) \in R_i} \rho(x, y)$ and $m_i \coloneqq \inf_{(x, y) \in R_i} \rho(x, y)$. Define
\[
J_{\bs{b}}^{K, L} \coloneqq \sum_{i = 1}^{K} \sqrt{M_i(b_i - b_{i-1}+1)\Delta x \Delta y}.
\]
\end{definition}

\begin{lemma}\label{L51}
\[
\varlimsup_{\substack{K \to \infty \\ L \to \infty}}\max_{\bs{b} \in \bs{B}_{KL}} J_{\bs{b}}^{K, L} \le \bar{J}
\]
where $\bar{J}$ is defined in Definition \ref{D51}.
\end{lemma}

\begin{proof}
Let $M$ be an upper bound of $\rho(x, y)$. In the context of Definition \ref{D52}, let $\phi_{\bs{b}}(x)$ be the piecewise linear function on $[0, 1]$ such that $\phi_{\bs{b}}(i \Delta x) = b_i \Delta y$, $i = 0, 1, \ldots, K$. From the two definitions above, we have
\begin{align}
J(\phi_{\bs{b}}) &= \int_0^1 \sqrt{\dot{\phi_{\bs{b}}}(x) \rho(x, \phi_{\bs{b}}(x))}\, dx \label{eq:L51a}\\
  &= \sum_{i =1}^{K} \int_{(i-1)\Delta x}^{i \Delta x} \sqrt{\dot{\phi_{\bs{b}}}(x) \rho(x, \phi_{\bs{b}}(x))}\, dx \nonumber\\
  &= \sum_{i = 1}^{K} \int_{(i-1)\Delta x}^{i \Delta x} \sqrt{\frac{(b_i - b_{i-1})\Delta y}{\Delta x}\cdot \rho(x, \phi_{\bs{b}}(x))}\, dx \nonumber\\
  &\ge \sum_{i = 1}^{K} \int_{(i-1)\Delta x}^{i \Delta x} \sqrt{\frac{(b_i - b_{i-1})\Delta y}{\Delta x}\cdot m_i}\, dx \nonumber\\
  &= \sum_{i = 1}^{K} \sqrt{m_i (b_i - b_{i - 1})\Delta x \Delta y} \nonumber\\
  &\ge \sum_{i = 1}^{K}\sqrt{M_i (b_i - b_{i - 1})\Delta x \Delta y} - \sum_{i = 1}^{K}\sqrt{(M_i-m_i) (b_i - b_{i - 1})\Delta x \Delta y}, \nonumber
\end{align}
where the last inequality follows since $\sqrt{a} + \sqrt{b} \ge \sqrt{a+b}$ for any $a, b \ge 0$. Moreover,
\begin{align}
  &\sum_{i = 1}^{K}\sqrt{M_i (b_i - b_{i - 1})\Delta x \Delta y} \label{eq:L51b}\\
  =\,& J_{\bs{b}}^{K, L} - \sum_{i = 1}^{K}\big(\sqrt{M_i (b_i - b_{i - 1}+1)\Delta x \Delta y} - \sqrt{M_i (b_i - b_{i - 1})\Delta x \Delta y}\,\big) \nonumber\\
  =\, & J_{\bs{b}}^{K, L} - \sum_{i = 1}^{K}\frac{M_i \Delta x \Delta y}{\sqrt{M_i (b_i - b_{i - 1}+1)\Delta x \Delta y} + \sqrt{M_i (b_i - b_{i - 1})\Delta x \Delta y}} \nonumber\\
  \ge\, & J_{\bs{b}}^{K, L} - \sum_{i = 1}^{K}\frac{M_i \Delta x \Delta y}{\sqrt{M_i (b_i - b_{i - 1}+1)\Delta x \Delta y}} \nonumber\\
  \ge\, & J_{\bs{b}}^{K, L} - \sum_{i = 1}^{K}\frac{M_i \Delta x \Delta y}{\sqrt{M_i \Delta x \Delta y}} \nonumber\\
  \ge\, & J_{\bs{b}}^{K, L} - \sqrt{M} \sum_{i = 1}^{K}\sqrt{\Delta x \Delta y} \nonumber\\
  =\, & J_{\bs{b}}^{K, L} - \sqrt{\frac{M}{L}}.  \nonumber
\end{align}
Next, define
\begin{align*}
D_1(\bs{b}) &\coloneqq \{i \in [K] : (b_i - b_{i-1}+1)\Delta y \le \sqrt[3]{\Delta x}\},\\
D_2(\bs{b}) &\coloneqq \{i \in [K] : (b_i - b_{i-1}+1)\Delta y > \sqrt[3]{\Delta x}\}.
\end{align*}
For $i \in D_1(\bs{b})$, the height of $R_i$ is no greater than $\sqrt[3]{\Delta x}$, and for $i \in D_2(\bs{b})$, the height of $R_i$ is greater than $\sqrt[3]{\Delta x}$. To bound the cardinality of $D_2(\bs{b})$, we have
\begin{align}
|D_2(\bs{b})|\sqrt[3]{\Delta x}\ &\le \sum_{i \in D_2(\bs{b})} (b_i - b_{i-1} +1) \Delta y  \label{eq:L51c}\\
 &\le \sum_{i \in D_2(\bs{b})} (b_i - b_{i-1}) \Delta y + |D_2(\bs{b})|\Delta y \nonumber\\
 &\le \sum_{i=1}^{K} (b_i - b_{i-1}) \Delta y + K \Delta y \nonumber\\
 &\le 1 + \frac{1}{L} \nonumber\\
 &\le 2.  \nonumber
\end{align}
Given $\epsilon > 0$, by the uniform continuity of $\rho(x, y)$ on $[0, 1]\times[0, 1]$, there exists $K_0 > 0$ such that, for any $K > K_0$ and any $i \in D_1(\bs{b})$, we have $M_i - m_i < \epsilon^2$. We can also choose $K_0$ sufficiently large such that, for any $K > K_0$,
\begin{equation}
2\sqrt{M}(\Delta x)^{\frac{1}{6}} < \epsilon .  \label{eq:L51d}
\end{equation}
Thus, for any $K > K_0$, we have
\begin{align}
  &\sum_{i = 1}^{K}\sqrt{(M_i-m_i) (b_i - b_{i - 1})\Delta x \Delta y} \label{eq:L51e}\\
\le\, &\sum_{i \in D_1(\bs{b})}\sqrt{\epsilon^2 (b_i - b_{i - 1})\Delta x \Delta y}\,\, + \sum_{i \in D_2(\bs{b})}\sqrt{M(b_i - b_{i - 1})\Delta x \Delta y}  \nonumber\\
\le\, &\epsilon \sum_{i=1}^{K}\sqrt{(b_i - b_{i - 1})\Delta x \Delta y} \,\, + \sum_{i \in D_2(\bs{b})}\sqrt{M \Delta x}  \nonumber\\
\le\, &\epsilon \sqrt{\textstyle\sum_{i=1}^{K}\Delta x}\sqrt{\textstyle\sum_{i=1}^{K}(b_i - b_{i - 1}) \Delta y} + 2\sqrt{M}(\Delta x)^{\frac{1}{6}} \nonumber\\
 <\, &\epsilon + \epsilon, \nonumber
\end{align}
where the second to last inequality follows by Cauchy-Schwarz inequality and (\ref{eq:L51c}).
Let $L_0 \coloneqq \left\lceil\frac{M}{\epsilon^2}\right\rceil$. By combining (\ref{eq:L51a}), (\ref{eq:L51b}) and (\ref{eq:L51e}), we get, for any $K > K_0$, $L>L_0$ and any $\bs{b}$,
\[
J_{\bs{b}}^{K, L} \le J(\phi_{\bs{b}}) + \sqrt{\frac{M}{L}} \le J(\phi_{\bs{b}}) +  3\epsilon \le \bar{J} +  3\epsilon,
\]
where the last inequality follows from the fact that $\phi_{\bs{b}} \in B_{\nearrow}$ and Definition \ref{D51}.
\end{proof}

\begin{definition}\label{D53}
In the context of Definition \ref{D52}, we call a sequence of points $(z_1, \ldots, z_m)$ with $z_i = (x_i, y_i)$ a $\bs{b}$-\emph{increasing sequence} if the following two conditions are satisfied.
\begin{itemize}
\item[(a)] $(z_1, \ldots, z_m)$ is an increasing sequence, that is $x_i < x_{i+1}$ and $y_i < y_{i+1}$ for all $i \in [m-1]$.
\item[(b)] Every point in the sequence is contained in some rectangle $R_j$ with $j\in[K]$. In other words, $(j-1)\Delta x < x_i \le j \Delta x$ implies $b_{j-1}\Delta y < y_i \le (b_j +1) \Delta y$.
\end{itemize}
Given a collection of points $\bs{z} = \{z_i\}_{i \in [n]}$, let $\lis_{\bs{b}}(\bs{z})$ denote the length of the longest $\bs{b}$-increasing subsequence of $\bs{z}$. That is
\begin{multline*}
\lis_{\bs{b}}(\bs{z}) \coloneqq \max\{m: \exists (i_1, i_2, \ldots, i_m)\\
 \qquad\qquad\qquad\quad\text{ such that } (z_{i_1}, z_{i_2}, \ldots, z_{i_m}) \text{ is a $\bs{b}$-increasing sequence}\}.
\end{multline*}
Note that we do not require $i_j < i_{j+1}$ above.
\end{definition}

\begin{lemma}\label{L52}
Under the same conditions as in Lemma \ref{L100}, for any $\delta > 0$, there exist $K_0$, $L_0$ such that, for any $K > K_0$, $L > L_0$ and any $\bs{b} = (b_0, b_1, \ldots, b_K) \in \bs{B}_{KL}$.
\begin{equation}
\lim_{n \to \infty}\p_n\left(\lis_{\bs{b}}(\bs{z}(\pi, \tau))> 2\sqrt{n}(\bar{J} + \delta)\right) = 0,  \label{eq:L52a}
\end{equation}
where $\bs{z}(\pi, \tau) \coloneqq \left\{\left(\frac{\pi(i)}{n}, \frac{\tau(i)}{n}\right)\right\}_{i \in [n]}$.
\end{lemma}

\begin{proof}
Given $\delta > 0$, by Lemma \ref{L51}, there exist $K_1, L_1 > 0$ such that, for any $K > K_1, L >L_1$ and any $\bs{b} = (b_0, b_1, \ldots, b_K) \in \bs{B}_{KL}$, we have
\[
J_{\bs{b}}^{K, L} < \bar{J} + \frac{\delta}{2}.
\]
Then, we get
\[
\p_n\left(\lis_{\bs{b}}(\bs{z}(\pi, \tau))> 2\sqrt{n}(\bar{J} + \delta)\right) \le \p_n\big(\lis_{\bs{b}}(\bs{z}(\pi, \tau))> 2\sqrt{n}(J_{\bs{b}}^{K, L} + \delta/2)\big).
\]
Hence, to show (\ref{eq:L52a}), it suffices to show that there exists $K_2$, $L_2$ such that, for any $K > K_2$, $L > L_2$ and any $\bs{b}$,
\begin{equation}
\lim_{n \to \infty}\p_n\left(\lis_{\bs{b}}(\bs{z}(\pi, \tau))> 2\sqrt{n}(J_{\bs{b}}^{K, L} + \delta/2)\right) = 0. \label{eq:L52b}
\end{equation}
Given $K, L > 0$, whose values are to be determined, and any $\bs{b} \in \bs{B}_{KL}$, we inherit all the notations introduced in Definition \ref{D52}. Let $l_{R_i}(\pi, \tau)$ denote the length of the longest increasing subsequence of $\bs{z}(\pi, \tau)$ wholly contained in the rectangle $R_i$. For any $i \in [K]$, define
\[
E_i(\bs{b}) \coloneqq \big\{(\pi, \tau): l_{R_i}(\pi, \tau) \ge 2\sqrt{n}\big(\sqrt{M_i(b_i - b_{i-1}+1)\Delta x \Delta y} + \delta \Delta x /2 \big)\big\}.
\]
Since $\lis_{\bs{b}}(\bs{z}(\pi, \tau)) \le \sum_{i = 1}^{K} l_{R_i}(\pi, \tau)$, we get
\[
\left\{\lis_{\bs{b}}(\bs{z}(\pi, \tau))> 2\sqrt{n}(J_{\bs{b}}^{K, L} + \delta/2)\right\} \subset \bigcup_{i\in [K]}E_i(\bs{b}).
\]
Hence, to show (\ref{eq:L52b}), it suffices to show
\begin{equation}
\lim_{n \to \infty}\p_n(E_i(\bs{b})) = 0,  \qquad \forall i \in [K] .  \label{eq:L52c}
\end{equation}
Let $M \coloneqq \sup_{0 \le x, y \le 1}\rho(x, y)$. Since $e^{\Delta x |\beta|/2} - 1 = \Theta(\Delta x)$, there exists $K_2 > 0$ such that, for any $K > K_2$, we have
\begin{equation}
e^{\Delta x |\beta|/2} < 1 + \frac{\delta \sqrt{\Delta x}}{2 \sqrt{M}} \quad \text{ and } \quad \Delta x |\beta| < \ln{2}.  \label{eq:L52d}
\end{equation}
Moreover, for any $i \in [K]$,
\begin{align}
 &\p_n(E_i(\bs{b})) \label{eq:L52e}\\
 \le\, &\textstyle\p_n\Big(l_{R_i}(\pi, \tau) \ge 2\sqrt{n}\sqrt{M_i(b_i - b_{i-1}+1)\Delta x \Delta y}\big(1 + \frac{\delta \Delta x}{2\sqrt{M\Delta x}} \big)\Big)  \nonumber\\
 \le\, &\textstyle\p_n\Big(l_{R_i}(\pi, \tau) \ge 2\sqrt{n \rho(R_i)}\big(1 + \frac{\delta \sqrt{\Delta x}}{2 \sqrt{M}} \big)\Big),  \nonumber
\end{align}
The first inequality follows since $(b_i - b_{i-1}+1)\Delta y \le 1$ and $M_i \le M$. The second inequality follows since
\[
M_i(b_i - b_{i-1}+1)\Delta x \Delta y \ge \int_{R_i} \rho(x, y)\,dxdy = \rho(R_i).
\]
Hence, combining (\ref{eq:L52d}), (\ref{eq:L52e}) and Lemma \ref{L501}, we get, for any $K > K_2$, $L > 0$ and any $\bs{b}$,
\[
\lim_{n \to \infty}\p_n(E_i(\bs{b})) = 0,  \qquad \forall i \in [K].
\]
Thus, (\ref{eq:L52c}) as well as the lemma follow.
\end{proof}


\begin{proof}[Proof of Theorem \ref{M3}]
By Proposition \ref{P3}, if $\pi \sim \m$, $\pi^{-1}$ has the same distribution $\m$. Hence, if $(\pi, \tau) \sim \m \times \mu_{n, q'}$, $(\pi^{-1}, \tau^{-1})$ has the same distribution $\m \times \mu_{n, q'}$. Note that $\lis(\pi, \tau) = \lis(\bs{z}(\pi, \tau))$. Thus, by Corollary \ref{C33}, to prove Theorem \ref{M3}, it suffices to show
\begin{equation}
\lim_{n \to \infty} \p_n \left(\left|\,\frac{\lis(\bs{z}(\pi, \tau))}{\sqrt{n}} - 2\bar{J} \, \right| < \epsilon \right) = 1, \label{eq:M3a}
\end{equation}
for any $\epsilon > 0$.
By Lemma \ref{L54} and the definition of $\bar{J}$, we have
\begin{equation}
\lim_{n \to \infty} \p_n \left(\frac{\lis(\bs{z}(\pi, \tau))}{\sqrt{n}} > 2\bar{J} - \epsilon \right) = 1. \label{eq:M3b}
\end{equation}
To show the upper bound in (\ref{eq:M3a}), note that, for any $K, L > 0$ and any increasing sequence of points $\{(x_j, y_j)\}_{j \in [n]}$ with $0 < x_j, y_j \le 1$, there exists a choice of $\bs{b}' = (b'_0, b'_1, \ldots, b'_K)$ such that $\{(x_j, y_j)\}_{j \in [n]}$ is a $\bs{b}'$ - increasing sequence. Specifically, we can define $\bs{b}'$ as follows. Let $\Delta x \coloneqq \frac{1}{K}, \Delta y \coloneqq \frac{1}{KL}$.
\begin{itemize}
\item Define $b'_0 \coloneqq 0$, $b'_K \coloneqq KL - 1$.
\item For $i \in [K-1]$, define $b'_i \coloneqq \lfloor\max{\{y_j : (i-1)\Delta x < x_j \le i \Delta x \}}\cdot KL\rfloor$.
\end{itemize}
It can be easily verified that with $\bs{b}'$ thus defined, every point $(x_j, y_j)$ is in some rectangle $R_i$, where $R_i$ is defined in Definition \ref{D52}. Hence, we get
\begin{align}
 &\p_n \left(\frac{\lis(\bs{z}(\pi, \tau))}{\sqrt{n}} > 2\bar{J} + \epsilon \right) \label{eq:M3c}\\
=\, &\p_n \left(\displaystyle\max_{\bs{b} \in \bs{B}_{KL}}{\big(\lis_{\bs{b}}(\bs{z}(\pi, \tau))\big)} > \sqrt{n}\big(2\bar{J} + \epsilon\big) \right) \nonumber\\
\le\, &\sum_{\bs{b} \in \bs{B}_{KL}}\p_n\left(\lis_{\bs{b}}(\bs{z}(\pi, \tau)) > \sqrt{n}\big(2\bar{J} + \epsilon\big) \right). \nonumber
\end{align}
By Lemma \ref{L52}, we can choose $K, L$ sufficiently large such that, for any $\bs{b} \in \bs{B}_{KL}$,
\[
\lim_{n \to \infty}\p_n\left(\lis_{\bs{b}}(\bs{z}(\pi, \tau)) > \sqrt{n}\big(2\bar{J} + \epsilon\big) \right) = 0.
\]
Hence, by (\ref{eq:M3c}) and the fact that the number of different choices of $\bs{b}$ is bounded above by $(KL)^K$, we have
\begin{equation}
\lim_{n \to \infty}\p_n\left(\frac{\lis(\bs{z}(\pi, \tau))}{\sqrt{n}} > 2\bar{J} + \epsilon \right) = 0. \label{eq:M3d}
\end{equation}
And (\ref{eq:M3a}) follows from (\ref{eq:M3b}) and (\ref{eq:M3d})
\end{proof}

\subsection{Solving $\bar{J}$ when $\beta = \gamma$}

The following lemma let us solve for the supremum $\bar{J}$ when the underlying density $\rho(x, y)$ satisfies $\rho\left(\frac{x+y}{2}, \frac{x+y}{2}\right) \ge \rho(x, y)$.
\begin{lemma}\label{L91}
Given a density $\rho(x, y)$ on $[0, 1]\times[0, 1]$ such that $\rho(x, y)$ is $C_b^1$ and $c < \rho(x, y) < C$ for some $C, c >0$, if $\rho(x, y) \le \rho\left(\frac{x+y}{2}, \frac{x+y}{2}\right)$ for any $0 \le x, y \le 1$, then we have
\[
\bar{J} = \int_0^1 \sqrt{\rho(x, x)} \, dx,
\]
i.\,e.\,the supremum of $J(\phi)$ on $B_{\nearrow}$ is attained for $\phi(x) = x$.
\end{lemma}

\begin{proof}
By the remark following Definition \ref{D51}, it suffices to show that, for any $\phi \in B^1_{\nearrow}$, we have
\begin{equation}
J(\phi) \le \int_0^1 \sqrt{\rho(x, x)} \, dx.  \label{eq:L91a}
\end{equation}
Define $g_{\phi}(x) \coloneqq x + \phi(x)$. Since $\dot{\phi}(x) \ge 0$, we have $\dot{g_{\phi}}(x) \ge 1$. Next, we reparameterize $\phi(x)$ as follows,
\begin{equation}
t\coloneqq \frac{g_{\phi}(x)}{2} = \frac{x + \phi(x)}{2}. \label{eq:L91b}
\end{equation}
Thus, we have $x = g_{\phi}^{-1}(2t)$ and $\phi(x) = 2t - x = 2t - g_{\phi}^{-1}(2t)$ where $t \in [0, 1]$. Moreover, since $g_{\phi}(x)$ is strictly increasing, $x$ is strictly increasing as a function of $t$. Hence we have
\begin{equation}
\rho\big(x, \phi(x)\big) = \rho\big(g_{\phi}^{-1}(2t), 2t - g_{\phi}^{-1}(2t)\big) \le \rho(t, t), \label{eq:L91c}
\end{equation}
where the last inequality follows since $\rho(x, y) \le \rho\left(\frac{x+y}{2}, \frac{x+y}{2}\right)$.
Next, by taking derivative with respect to $t$ on both sides of (\ref{eq:L91b}), we have
\begin{equation}
1 = \frac{1}{2}\left(\frac{d{x}}{d{t}} + \dot{\phi}(x)\, \frac{d{x}}{d{t}}\right). \label{eq:L91d}
\end{equation}
By multiplying $2 \, \frac{d{x}}{d{t}}$ on both sides of (\ref{eq:L91d}), we get
\begin{equation}
\dot{\phi}(x)\,\left(\frac{d{x}}{d{t}}\right)^2 = 2\,\frac{d{x}}{d{t}} - \left(\frac{d{x}}{d{t}}\right)^2 \le 1. \label{eq:L91e}
\end{equation}
Hence, by (\ref{eq:L91c}) and (\ref{eq:L91e}), we have
\begin{align*}
J(\phi) &= \int_0^1 \sqrt{\dot{\phi}(x)\, \rho(x, \phi(x))} \, dx\\
        &\le \int_0^1 \sqrt{\dot{\phi}(x)\, \rho(t, t)} \cdot \frac{d{x}}{d{t}} \,dt\\
        & = \int_0^1 \sqrt{\rho(t, t)\, \dot{\phi}(x)\left(\frac{d{x}}{d{t}}\right)^2} \,dt\\
        &\le \int_0^1 \sqrt{\rho(t, t)}\, dt.
\end{align*}
Therefore, $\bar{J}$ is attained for $\phi(x) = x$.
\end{proof}

\begin{proof}[Proof of Corollary \ref{C92}]
Note that in the special case where $\beta = \gamma$, the density $\rho(x, y)$ in (\ref{eq:L100z}) is given by
\begin{equation}
\rho(x, y) \coloneqq \int_0^1 u(x, t, \beta)\cdot u(t, y, \beta)\, dt .\label{eq:C92a}
\end{equation}
In this case, we will show that $\rho(x, y) \le \rho\left(\frac{x+y}{2}, \frac{x+y}{2}\right)$ for any $0 \le x, y \le 1$. Hence, by Lemma \ref{L101} and Lemma \ref{L91}, $\bar{J}$ defined in Theorem \ref{M3} is attained when $\phi(x) = x$. In fact, by direct calculation, it can be shown that
\begin{equation}
u(x, t, \beta)\cdot u(t, y, \beta) \le u\left(\frac{x+y}{2}, t, \beta\right)\cdot u\left(t, \frac{x+y}{2}, \beta\right), \label{eq:C92b}
\end{equation}
for any $0 \le x, y, t \le 1$.
By the definition of $u(x, y, \beta)$, we have
\begin{align}
  &u(x, t, \beta)\cdot u(t, y, \beta) \label{eq:C92c}\\
=\, &\frac{(\beta/2) \sinh(\beta/2)}{\left(e^{\beta/4} \cosh(\beta[x-t]/2)-e^{-\beta/4}\cosh(\beta[x+t-1]/2)\right)^2} \nonumber\\
  &\qquad\qquad \times\frac{(\beta/2) \sinh(\beta/2)}{\left(e^{\beta/4} \cosh(\beta[t-y]/2)-e^{-\beta/4}\cosh(\beta[t+y-1]/2)\right)^2}\nonumber\\
=\, &\frac{\beta (e^{\beta} - 1)}{\left(2e^{\beta/2} \cosh(\beta[x-t]/2)-2\cosh(\beta[x+t-1]/2)\right)^2}\nonumber\\
  &\qquad\qquad \times\frac{\beta (e^{\beta} - 1)}{\left(2e^{\beta/2} \cosh(\beta[t-y]/2)-2\cosh(\beta[t+y-1]/2)\right)^2}. \nonumber
\end{align}
Considering the term inside the square of the denominator, by using the hyperbolic trigonometric identities,
\begin{align*}
\cosh(x)\cosh(y) &= \big(\cosh(x+y) + \cosh(x-y) \big)/2,   \\
\cosh(x + y) &= \cosh(x)\cosh(y) + \sinh(x)\sinh(y),  \\
\cosh(x - y) &= \cosh(x)\cosh(y) - \sinh(x)\sinh(y),
\end{align*}
we get
\begin{align}
  &\big(2e^{\beta/2} \cosh(\beta[x-t]/2)-2\cosh(\beta[x+t-1]/2)\big) \label{eq:C92cc}\\
   &\qquad\qquad \times\big(2e^{\beta/2} \cosh(\beta[t-y]/2)-2\cosh(\beta[t+y-1]/2)\big) \nonumber\\
=\, &2e^{\beta}\big(\cosh(\beta[x-y]/2) + \cosh(\beta[x+y-2t]/2)\big)  \nonumber\\
   &\qquad - 2e^{\beta/2}\big(\cosh(\beta[x+y-1]/2) + \cosh(\beta[x-y-2t+1]/2)\big)  \nonumber\\
   &\qquad - 2e^{\beta/2}\big(\cosh(\beta[x-y+2t-1]/2) + \cosh(\beta[x+y-1]/2)\big)  \nonumber\\
   &\qquad + 2 \big(\cosh(\beta[x+y+2t-2]/2) + \cosh(\beta[x-y]/2)\big) \nonumber\\
=\, & S^-_t + S^+_t, \nonumber
\end{align}
where $S^-_t$ denotes the sum of those terms in the above equation containing the term $x-y$ and $S^+_t$ denotes the sum of those which contain the term $x+y$. After further simplification using the identities above, we have
\begin{equation}
S^-_t = 2\cosh(\beta[x-y]/2)\big(e^{\beta} - 2e^{\beta/2}\cosh(\beta[2t-1]/2) +1\big). \label{eq:C92d}
\end{equation}
It is easily seen that the minimum of $e^{\beta} - 2e^{\beta/2}\cosh(\beta[2t-1]/2) +1$ for $0 \le t \le 1$ is attained when $t = 0, 1$, and the minimum is 0. Hence, for any $t \in [0, 1]$, $S^-_t$ is minimized when $x = y$. Thus to prove (\ref{eq:C92b}), it suffices to show that $S^+_t \ge 0$, since $S^-_t + S^+_t$ is the term inside the square of the denominator of (\ref{eq:C92c}). After simplification, we have
\begin{align}
S^+_t &= 2e^{\beta}\Big( \cosh(\beta[x+y-1]/2)\cosh(\beta[2t-1]/2)  \label{eq:C92e}\\
      &\qquad\qquad\qquad - \sinh(\beta[x+y-1]/2)\sinh(\beta[2t-1]/2)\Big) \nonumber\\
      &\qquad -4e^{\beta/2}\cosh(\beta[x+y-1]/2) \nonumber\\
      &\qquad + 2\Big( \cosh(\beta[x+y-1]/2)\cosh(\beta[2t-1]/2)  \nonumber\\
      &\qquad\qquad\qquad + \sinh(\beta[x+y-1]/2)\sinh(\beta[2t-1]/2)\Big). \nonumber
\end{align}
Next, we make change of variables. Define $r \coloneqq e^{\beta(x+y-1)/2}$, $s \coloneqq e^{\beta(2t-1)/2}$. Then, from (\ref{eq:C92e}), we have
\begin{align}
S^+_t &= \frac{e^{\beta}}{2}\Big(\Big(r + \frac{1}{r}\Big)\Big(s +\frac{1}{s}\Big) - \Big(r - \frac{1}{r}\Big)\Big(s - \frac{1}{s}\Big)\Big)
          - 2e^{\beta/2}\Big(r + \frac{1}{r}\Big) \nonumber\\
      &\qquad\qquad + \frac{1}{2}\Big(\Big(r + \frac{1}{r}\Big)\Big(s +\frac{1}{s}\Big) + \Big(r - \frac{1}{r}\Big)\Big(s - \frac{1}{s}\Big)\Big) \nonumber\\
      &= e^{\beta}\Big(\frac{r}{s} + \frac{s}{r}\Big) - 2e^{\beta/2}\Big(r + \frac{1}{r}\Big) + \Big(rs + \frac{1}{rs}\Big) \label{eq:C92f}\\
      &= \Big(\frac{e^{\beta}r}{s} + rs - 2e^{\beta/2}r\Big) + \Big(\frac{e^{\beta}s}{r} + \frac{1}{rs} - \frac{2e^{\beta/2}}{r}\Big)\nonumber\\
      &\ge 0, \nonumber
\end{align}
where the last inequality follows since $ x + y \ge 2\sqrt{xy}$ for any $x, y\ge 0$.
We complete the proof of Corollary \ref{C92} by showing:
\begin{equation}
\int_0^1 u(x, t, \beta)\cdot u(t, x, \beta)\,dt = \frac{\beta \big(\cosh(\beta/2) + 2 \cosh\big(\beta[2x - 1]/2\big)\big)}{6 \sinh{(\beta/2)}}, \label{eq:C92g}
\end{equation}
for $0 \le x \le 1$.

By the same change of variables as above, since $y = x$, let $r \coloneqq e^{\beta(2x-1)/2}$, $s \coloneqq e^{\beta(2t-1)/2}$. Then, we have
\begin{equation}
\frac{dt}{ds} = \frac{1}{\frac{ds}{dt}} = \frac{1}{s\beta}. \label{eq:C92h}
\end{equation}
By (\ref{eq:C92d}), we have,
\begin{equation}
S^-_t = 2\left(e^{\beta} - e^{\beta/2}\Big(s + \frac{1}{s}\Big) + 1\right). \label{eq:C92i}
\end{equation}
Then, by (\ref{eq:C92f}) and (\ref{eq:C92i}), it can be easily verified that
\begin{equation}
rs\left(S^+_t + S^-_t\right) = \left(e^{\beta/2}(r + s) - (rs +1)\right)^2 . \label{eq:C92j}
\end{equation}
Hence, we have
\begin{align}
   &\int_0^1 u(x, t, \beta)\cdot u(t, x, \beta)\,dt \label{eq:C92k}\\
=\,&\int_{e^{-\beta/2}}^{e^{\beta/2}} \frac{\beta^2(e^{\beta}-1)^2}{\left(S^+_t + S^-_t\right)^2}\frac{1}{s\beta}\,ds \nonumber\\
=\,&\int_{e^{-\beta/2}}^{e^{\beta/2}} \frac{\beta(e^{\beta}-1)^2r^2s}{\left(rs\left(S^+_t + S^-_t\right)\right)^2}\,ds \nonumber\\
=\,&\int_{e^{-\beta/2}}^{e^{\beta/2}} \frac{\beta(e^{\beta}-1)^2r^2s}{\left(e^{\beta/2}(r + s) - (rs +1)\right)^4}\,ds \nonumber\\
=\,&\beta(e^{\beta}-1)^2r^2 \int_{e^{-\beta/2}}^{e^{\beta/2}} \frac{s}{\left((e^{\beta/2}-r)s + e^{\beta/2}r - 1\right)^4}\,ds \nonumber \\
=\, &\beta(e^{\beta} - 1)^2 e^{\beta(2x-1)} \int_{e^{-\beta/2}}^{e^{\beta/2}} \frac{s}{(e^{\beta/2}(1 - e^{\beta(x-1)})s + e^{\beta x} -1)^4}\,ds.   \nonumber
\end{align}
The first equality above follows from (\ref{eq:C92c}), (\ref{eq:C92cc}), (\ref{eq:C92h}) and change of variables. The third equality follows from (\ref{eq:C92j}).
Then we make another change of variable by defining
\[
w \coloneqq \frac{e^{\beta/2}(1 - e^{\beta(x-1)})s + e^{\beta x} -1}{e^{\beta} - 1},
\]
from which we have
\[
\frac{ds}{dw} = \frac{e^{\beta} - 1}{e^{\beta/2}(1 - e^{\beta(x-1)})}, \quad \text{ and } \quad w = 
\begin{cases}
1 &\text{when } s = e^{\beta/2},\\
e^{\beta(x-1)} &\text{when } s = e^{-\beta/2}.
\end{cases}
\]
Hence, by (\ref{eq:C92k}), we have
\begin{align}
&\int_0^1 u(x, t, \beta)\cdot u(t, x, \beta)\,dt \nonumber\\
=\,& \frac{\beta\ e^{2\beta(x-1)}}{(e^{\beta}-1)(1 - e^{\beta(x-1)})^2} \int_{e^{\beta(x-1)}}^1 \frac{(e^{\beta}-1)w - e^{\beta x} + 1}{w^4}\,dw \nonumber\\
=\,& \frac{\beta\ e^{2\beta(x-1)}}{(e^{\beta}-1)(1 - e^{\beta(x-1)})^2} \left(\frac{1 - e^{\beta}}{2w^2} + \frac{e^{\beta x} - 1}{3w^3}\right)\bigg\vert_{e^{\beta(x-1)}}^1 \nonumber\\
=\,& \frac{\beta\left(1 + e^{\beta} + 2e^{\beta x} + 2e^{-\beta(x-1)}\right)}{6(e^{\beta} - 1)} \nonumber\\
=\,& \frac{\beta \left(\cosh(\beta/2) + 2 \cosh\big(\beta[2x - 1]/2\big)\right)}{6 \sinh{(\beta/2)}}. \nonumber  \qedhere
\end{align}
\end{proof}

\section{Proof of Lemma \ref{L40}}

To prove Lemma \ref{L40}, we use the same techniques developed in the proof of Corollary 4.3 in \cite{MuellerStarr}, in which the authors constructed a coupling of two point processes. A point process is a random, locally finite, nonnegative integer valued measure. Let $\mathcal{X}_k$ denote the set of all Borel measures $\xi$ on $\mathbb{R}^k$ such that $\xi(A) \in \{0, 1, 2, \ldots \}$ for any bounded Borel set $A$ in $\mathbb{R}^k$. Then, a point process on $\mathbb{R}^k$ is a random variable which takes value in $\mathcal{X}_k$.

Suppose $\mu, \nu$ are two measures on $\mathbb{R}^k$. We say $\mu \le \nu$ if $\mu(A) \le \nu(A)$ for any $A \in \mathcal{B}(\mathbb{R}^k)$.

\begin{lemma}\label{L43}
Suppose $\hat{\alpha}$ and $\alpha$ are two probability measures on $[0, 1]$ with density $f(x)$, $g(x)$ respectively. If, for any $x\in[0,1]$, $f(x) \ge p \cdot g(x)$ for some $0 < p < 1$, then there exist random variables $X$, $Y$ and $B_p$ such that the following hold.
\begin{itemize}
\item $X$ is $\hat{\alpha}$-distributed, $Y$ is $\alpha$-distributed and $B_p$ is Bernoulli distributed with $\p(B_p = 1) = p$.
\item $B_p$ and $Y$ are independent.
\item Define two point processes $\eta$, $\xi$ on $[0, 1]$ as follows,
\[
\xi(A) \coloneqq \mathds{1}_A(X) \quad \text{and} \quad \eta(A) \coloneqq B_p \cdot \mathds{1}_A(Y), \qquad \forall A \in \mathcal{B}([0, 1]).
\]
Then, we have $\eta \le \xi$.
\end{itemize}
\end{lemma}

\begin{proof}
Let $Y$, $Y'$ and $B_p$ be independent random variables defined on the same probability space such that $Y$ is $\alpha$-distributed, $B_p$ is Bernoulli distributed with $\p(B_p = 1) = p$ and the density of the distribution of $Y'$ is $\frac{f(x) - p \cdot g(x)}{1 - p}$. Define $X \coloneqq B_p Y + (1 - B_p)Y'$. To see that $X$ thus defined is $\hat{\alpha}$-distributed, we have
\[
\p(X \in A) = p \int_A g(x)\,dx + (1 - p) \int_A \frac{f(x) - p \cdot g(x)}{1 - p}\,dx = \int_A f(x)\,dx,
\]
for any $A \in \mathcal{B}([0, 1])$.
Finally, the two point processes $\xi$ and $\eta$ thus defined satisfy $\eta \le \xi$, since for any $A \in \mathcal{B}([0, 1])$, when $B_p = 1$, we have $\xi(A) = \eta(A)$, and, when $B_p = 0$, we have $\eta(A) = 0$.
\end{proof}

\begin{lemma}\label{L44}
Suppose $\hat{\alpha}$ and $\alpha$ are two probability measures on $[0, 1]$ with density $f(x)$, $g(x)$ respectively. If $(1 - \theta_1)g(x) \le f(x) \le (1 + \theta_2) g(x)$ for some $\theta_1, \theta_2 \ge 0$ with $\theta_1 + \theta_2 < 1$, then there exist random variables $X$, $Y$, $Z$ and $B_{\theta}$ such that the following hold.
\begin{itemize}
\item $X$ is $\hat{\alpha}$-distributed, $Y$ and $Z$ are $\alpha$-distributed and $B_{\theta}$ is Bernoulli distributed with $\p(B_{\theta} = 1) = \theta$, where $ \theta = \theta_1 + \theta_2$.
\item $B_{\theta}$, $Y$ and $Z$ are independent.
\item Define two point processes $\xi$, $\zeta$ on $[0, 1]$ as follows,
\[
\xi(A) \coloneqq \mathds{1}_A(X) \quad \text{and} \quad \zeta(A) \coloneqq \mathds{1}_A(Y) + B_{\theta} \cdot \mathds{1}_A(Z), \qquad \forall A \in \mathcal{B}([0, 1]).
\]
Then, we have $\xi \le \zeta$.
\end{itemize}
\end{lemma}

\begin{proof}
Let $Y$, $Z$ and $B_{\theta}$ be independent random variables defined on the same probability space such that $Y$, $Z$ is $\alpha$-distributed, $B_{\theta}$ is Bernoulli distributed with $\p(B_{\theta} = 1) = \theta$. We define a new random variable $X$ as follows. Conditioned on $Y = y$ and $Z = z$,
\begin{itemize}
\item if $B_{\theta} = 0$, define $X = y$.
\item If $B_{\theta} = 1$, flip a coin $W$ with probability of heads being $\frac{f(z) - (1 - \theta_1)g(z)}{\theta\cdot g(z)}$. If the result is head, define $X = z$, else define $X = y$.
\end{itemize}
Note that, without loss of generality, here we may assume $g(z) > 0$, since $\p(g(Z) = 0) = 0$. It is straight forward that the two point processes $\xi$ and $\zeta$ thus defined satisfy $\xi \le \zeta$. We complete the proof by verifying that $X$ thus defined has distribution $f(x)$.\\
For any $A \in \mathcal{B}([0, 1])$, the event $\{X \in A\}$ can be partitioned into three parts: $\{B_{\theta} = 0, Y \in A\}$,
$\{B_{\theta} = 1, W \text{ is head }, Z \in A\}$ and $\{B_{\theta} = 1, W \text{ is tail }, Y \in A\}$. We have
\begin{align*}
\p(\{B_{\theta} = 0, Y \in A\}) = &\ \textstyle(1 - \theta) \int_A g(x)\,dx = (1 - \theta)\alpha(A),\\
\p(\{B_{\theta} = 1, W \text{ is head}, Z \in A\})
 = &\ \textstyle\theta \int_A \frac{f(z) - (1 - \theta_1)g(z)}{\theta\cdot g(z)}g(z)\,dz\\
 = &\ \textstyle\int_A f(z)\,dz - (1 - \theta_1)\alpha(A),\\
\p(\{B_{\theta} = 1, W \text{ is tail}, Y \in A\})
 = &\ \textstyle\theta \int_A g(y)\,dy \int_0^1\big(1 -\frac{f(z) - (1 - \theta_1)g(z)}{\theta\cdot g(z)}\big)g(z)\,dz\\
 = &\ \textstyle\alpha(A) \int_0^1 (1 + \theta_2)g(z) - f(z)\,dz\\
 = &\ \textstyle\alpha(A)\,\theta_2.
\end{align*}
Here we evaluate the last two probabilities by conditioning on the value of $Z$. Summing up the three probabilities, we get
\[
\textstyle\p(\{X \in A\}) = \int_A f(z)\,dz. \qedhere
\]
\end{proof}

Next, we define a triangular array of random variables in $[0, 1]$.
\begin{definition}\label{D4}
Suppose that $\{q_n\}_{n=1}^{\infty}$ is a sequence such that $q_n > 0$. For any $n \in \mathbb{N}$, we define the random vector $(Y_1^{(n)}, Y_2^{(n)}, \ldots, Y_n^{(n)})$ as follows. Let $\{Y_i\}_{i=1}^n$ be i.i.d.\,uniform random variables on $[0, 1]$. Let $\{Y_{(i)}\}_{i = 1}^n$ be the order statistics of $\{Y_i\}_{i=1}^n$. Independently, let $\pi$ be a $\n$-distributed random variable on $S_n$. We define $Y_i^{(n)} \coloneqq Y_{(\pi(i))}$ for all $i \in [n]$.
\end{definition}

In the remainder of this paper, we use $(Y_1^{(n)}, \ldots, Y_n^{(n)})$ specifically to denote the random vector defined as above.
Next, we define the function $\Phi$ which maps vectors in $\mathbb{R}^n$ or $n$ points in $\mathbb{R}^2$ to the induced permutation in $S_n$.
\begin{definition}\label{D41}
Suppose $\bs{x} = (x_1, x_2, \ldots, x_n)$ is a vector in $\mathbb{R}^n$ such that all its entries are distinct. Let $\Phi(\bs{x})$ denote the permutation in $S_n$ such that, for any $i \in [n]$, $\Phi(\bs{x})(i) = j$\,\ if $x_i$ is the $j$-th smallest entry in $\bs{x}$. Similarly, suppose $\bs{z} = \{(x_i, y_i)\}_{i = 1}^{n}$ are $n$ points in $\mathbb{R}^2$ such that they share no $x$ coordinate nor any $y$ coordinate. Let $\Phi(\bs{z})$ denote the permutation in $S_n$ such that, for any $i \in [n]$, $\Phi(\bs{z})(i) = j$\,\ if there exits $k \in [n]$, such that $x_k$ is the $i$-th smallest term in $\{x_i\}_{i = 1}^{n}$ and $y_k$ is the $j$-th smallest term in $\{y_i\}_{i = 1}^{n}$.
\end{definition}
\begin{remark*}
From the above definitions, it can be easily seen that
\begin{itemize}
\item[(a)] For any $x_1 < \cdots < x_n$ and $\bs{y} = (y_1, \ldots, y_n) \in \mathbb{R}^n$, we have $\Phi(\bs{y}) = \Phi(\{(x_i, y_i)\}_{i = 1}^n )$.
\item[(b)] For any $\bs{y} = (y_1, \ldots, y_n) \in \mathbb{R}^n$ and increasing indices $\bs{b} = (b_1, \ldots, b_m)$, we have $\Phi(\bs{y})_{\bs{b}} = \Phi((y_{b_1}, \ldots, y_{b_m}))$.
\item[(c)] $\Phi((Y_1^{(n)}, Y_2^{(n)}, \ldots, Y_n^{(n)}))$ is $\n$-distributed.
\end{itemize}
\end{remark*}

Let $D_n$ be the set of vectors in $[0, 1]^n$ which contain (at least two) identical entries. It is not hard to show that the density function of $(Y_1^{(n)}, \ldots, Y_n^{(n)})$ is the following,
\[
f_n(\bs{y}) =\n(\Phi(\bs{y}))\cdot n\,! \qquad \text{for all } \bs{y} \in [0, 1]^n \setminus D_n.
\]
Since $\{Y_i^{(n)}\}_{i = 1}^{n} = \{Y_i\}_{i = 1}^{n}$ are $n$ i.i.d.\,uniform samples from $[0, 1]$, we have $\p((Y_1^{(n)}, \ldots, Y_n^{(n)}) \in D_n) = 0$. Intuitively, for any $0 \le y_1 < \cdots < y_n \le 1$, there are $n!$ ways to choose the vector $(Y_1, \ldots, Y_n)$ such that $\{Y_i\}_{i=1}^n = \{y_i\}_{i=1}^n$. Moreover, conditioned on $\{Y_i\}_{i=1}^n = \{y_i\}_{i=1}^n$, the probability of $(Y_1^{(n)}, \ldots, Y_n^{(n)}) = (y_{\pi(1)}, \ldots, y_{\pi(n)})$ is $\n(\pi)$. Since the measure of $D_n$ is zero, when $\bs{y} \in D_n$, we can define $f_n(\bs{y})$ to be an arbitrary value.

\begin{lemma}\label{L47}
Given $i \in [n]$ and a vector $(y_1, \ldots, y_{i-1}, y_{i+1}, \ldots, y_n) \in [0, 1]^{n-1}\setminus D_{n-1}$, let $\hat{\alpha}$ denote the distribution of $\,Y_i^{(n)}$ conditioned on the event $\{Y_j^{(n)} = y_j \text{ for all } j\in [n]\setminus \{i\} \}$. Then $\hat{\alpha}$ has density $f(y)$ on $[0, 1]$ such that, for any $y, y' \in [0, 1]\setminus \{y_1, \ldots, y_{i-1}, y_{i+1}, \ldots, y_n\}$, we have
\[
f(y) \ge \min{\left(q_n^n, \frac{1}{q_n^n}\right)}, \qquad f(y) - f(y') \le \max{\left(q_n^n, \frac{1}{q_n^n}\right)} - 1.
\]
\end{lemma}
\begin{proof}
Since $(Y_1^{(n)}, \ldots, Y_n^{(n)})$ has density $f_n(\bs{y}) =\n(\Phi(\bs{y}))\cdot n\,!$ on $[0, 1]^n \setminus D_n$, the density $f(y)$ of $\hat{\alpha}$ is given by
\[
f(y) = \frac{\n(\Phi((y_1, \ldots, y_{i-1}, y, y_{i+1}, \ldots, y_n)))}
{\int_0^1\n(\Phi((y_1, \ldots, y_{i-1}, t, y_{i+1}, \ldots, y_n)))\,dt},
\]
for any $y \in [0, 1]/\{y_1, \ldots, y_{i-1}, y_{i+1}, \ldots, y_n\}$.
It can be seen from the definition that $f(y)$ is a simple function which takes at most $n$ different values. Let $M$ and $m$ denote the maximum and minimum of $f(y)$ respectively. Then we have $M \ge 1$ and $0 < m \le 1$. Moreover, for any $y, y'\in [0, 1]$, let $\bs{y} \coloneqq (y_1, \ldots, y_{i-1}, y, y_{i+1}, \ldots, y_n)$ and $\bs{y}' \coloneqq (y_1, \ldots, y_{i-1}, y', y_{i+1}, \ldots, y_n)$. We have
\[
|\inv(\Phi(\bs{y})) - \inv(\Phi(\bs{y}'))| \le n - 1.
\]
That is, if $\bs{y}$ and $\bs{y}'$ differ at one entry, the number of inversions of the induced permutations differ at most by $n-1$. Hence, assuming $q_n \ge 1$, for any $y, y'\in [0, 1]$, we have
\[
\frac{1}{q_n^{n-1}} \le \frac{f(y)}{f(y')} \le q_n^{n-1}.
\]
Choose $y'$ such that $f(y') = M$, we have $f(y) \ge M/q_n^{n-1} \ge 1/q_n^n$. For the second part, we choose $y, y'$ such that $f(y) = M$ and $f(y') = m$. Then we have $M/m - 1 \le q_n^{n-1} - 1 \le q_n^n - 1$. Thus, $M - m \le q_n^n - 1$, since $0 < m \le 1$.
The argument for the case when $0 < q_n < 1$ is similar.
\end{proof}

\begin{lemma}\label{L45}
Given $n \in \mathbb{N}$ and $q_n > 0$, for any $m \le n$ and any increasing indices $\bs{b} = (b_1, \ldots, b_m)$, there exists a random vector $(V_1, \ldots, V_n) \in [0, 1]^n$ and $2m$ independent random variables $\{U_i\}_{i =1}^{m} \cup \{B_i\}_{i = 1}^{m}$ such that $(V_1, \ldots, V_n)$ has the same distribution as $(Y_1^{(n)}, \ldots, Y_n^{(n)})$, each $U_i$ is uniformly distributed on $[0, 1]$ and each $B_i$ is a Bernoulli random variable with $\p(B_i = 1) = \min(q_n^n, 1/q_n^n)$. Moreover, if we define two point processes as follows,
\[
\xi_{\bs{b}}^{(n)}(A) \coloneqq \sum_{i = 1}^{m}\mathds{1}_A((i, V_{b_i})), \quad
\eta_m(A) \coloneqq \sum_{i = 1}^{m}B_i\cdot\mathds{1}_A((i, U_i)), \ \forall A \in \mathcal{B}(\mathbb{N}\times[0, 1]),
\]
we have $\eta_m \le \xi_{\bs{b}}^{(n)}$ almost surely.
\end{lemma}

\begin{proof}
Given $n$, $m$ and $\bs{b}$, let $(Y_1^{(n)}, \ldots, Y_n^{(n)})$ be as defined in Definition \ref{D4} and, independently, define $2m$ independent random variables $\{U_i\}_{i =1}^{m} \cup \{B_i\}_{i = 1}^{m}$ such that each $U_i$ is uniformly distributed on $[0, 1]$ and each $B_i$ is a Bernoulli random variable with $\p(B_i = 1) = \min(q_n^n, 1/q_n^n)$. We define the random vector $(V_1, \ldots, V_n)$ as follows,
\begin{itemize}
\item Sample the random vector $(Y_1^{(n)}, \ldots, Y_n^{(n)})$, say, we get\\
 $(Y_1^{(n)}, \ldots, Y_n^{(n)}) = (y_1, \ldots, y_n)$.
\item For $j \in [n]\setminus\{b_i\}_{i = 1}^m$, let $V_j \coloneqq y_j$.
\item For each $i \in [m]$, we resample $Y_{b_i}^{(n)}$ one by one, conditioned on the current value of other $Y_j^{(n)}$. Let $y'_{b_i}$ denote the new value of $Y_{b_i}^{(n)}$ after the resampling and define $V_{b_i} \coloneqq y'_{b_i}$ . Specifically, for each $i \in [m]$, we sample a value $y'_{b_i}$ according to the distribution of $Y_{b_i}^{(n)}$, conditioned on the event 
\[
\left\{Y_{b_j}^{(n)} = y'_{b_j} \text{ for }\forall j < i \text{ and } Y_k^{(n)} = y_k \text{ for }\forall k \in [n]\setminus\{b_j\}_{j\in[i]}\right\}.
\]
\item In each resampling step, say, resampling $Y_{b_i}^{(n)}$, let $\hat{\alpha}$ denote the above conditional distribution of $Y_{b_i}^{(n)}$. By Lemma \ref{L47}, we know that that $\hat{\alpha}$ has density $f(y)$ with $f(y) \ge \min(q_n^n, 1/q_n^n)$. Hence, we can couple this resampling procedure with variables $U_i$ and $B_i$ in the same fashion as in the proof of Lemma \ref{L43}, with $\alpha$ in that lemma being the uniform measure on $[0, 1]$. Thus we have $\mathds{1}_A((i, V_{b_i})) \ge B_i\cdot\mathds{1}_A((i, U_i))$ a.\,s.\, for any $A \in \mathcal{B}(\mathbb{N}\times[0, 1])$.
\end{itemize}
It can be easily seen from the above procedure that $(V_1, \ldots, V_n)$ thus defined has the same distribution as $(Y_1^{(n)}, \ldots, Y_n^{(n)})$, and
\[
\eta_m(A) = \sum_{i = 1}^{m}B_i\cdot\mathds{1}_A((i, U_i)) \le \sum_{i = 1}^{m}\mathds{1}_A((i, V_{b_i})) = \xi_{\bs{b}}^{(n)}(A) \text{ a.\,s.\,}
\]
for any $A \in \mathcal{B}(\mathbb{N}\times[0, 1])$.
\end{proof}

\begin{lemma}\label{L46}
Given $n \in \mathbb{N}$ and $q_n > 0$ such that $\max(q_n^n, 1/q_n^n) < 2$, for any $m \le n$ and any increasing indices $\bs{b} = (b_1, \ldots, b_m)$, there exists a random vector $(V_1, \ldots, V_n) \in [0, 1]^n$ and $3m$ independent random variables $\{U_i, U'_i, B_i\}_{i = 1}^{m}$ such that $(V_1, \ldots, V_n)$ has the same distribution as the vector $(Y_1^{(n)}, \ldots, Y_n^{(n)})$, each $U_i, U'_i$ are uniformly distributed on $[0, 1]$ and each $B_i$ is a Bernoulli random variable with $\p(B_i = 1) = \max(q_n^n, 1/q_n^n) - 1$. Moreover, if we define two point processes as follows,
\begin{align*}
\xi_{\bs{b}}^{(n)}(A) &\coloneqq \sum_{i = 1}^{m}\mathds{1}_A((i, V_{b_i})),& &\forall A \in \mathcal{B}(\mathbb{N}\times[0, 1])\\
\zeta_m(A) &\coloneqq \sum_{i = 1}^{m}\mathds{1}_A((i, U'_i)) + B_i\cdot\mathds{1}_A((i, U_i)),& &\forall A \in \mathcal{B}(\mathbb{N}\times[0, 1])
\end{align*}
we have $\xi_{\bs{b}}^{(n)} \le \zeta_m$ almost surely.
\end{lemma}

\begin{proof}
The proof of this lemma is similar to the proof of Lemma \ref{L45}.
Given $n$, $m$ and $\bs{b}$, define $(Y_1^{(n)}, \ldots, Y_n^{(n)})$ as in Definition \ref{D4} and, independently, define $3m$ independent random variables $\{U_i, U'_i, B_i\}_{i = 1}^{m}$ such that each $U_i, U'_i$ are uniformly distributed on $[0, 1]$ and each $B_i$ is a Bernoulli random variable with $\p(B_i = 1) = \max(q_n^n, 1/q_n^n) - 1$. Then we define the random vector $(V_1, \ldots, V_n)$ by the same steps as in the proof of Lemma~\ref{L45}, except that, in each resampling step, we couple the resampling of $Y_{b_i}^{(n)}$ with the variables $U_i, U'_i$ and $B_i$ in the same way as in the proof of Lemma~\ref{L44}, with $\alpha$ in that lemma being the uniform measure on $[0, 1]$. Note that the second inequality in Lemma~\ref{L47} ensures that the conditions in Lemma~\ref{L44} are met. Specifically, in each resampling step, let $f(y)$ denote the density of the conditional distribution of $Y_{b_i}^{(n)}$. Let $M$, $m$ be the maximum and minimum of $f(y)$ respectively. Define $\theta_1 \coloneqq 1 - m$ and $\theta_2 \coloneqq M - 1$. Hence, $1 - \theta_1 \le f(y) \le 1 + \theta_2$ almost surely and $\theta_1 + \theta_2 = M - m \le \max(q_n^n, 1/q_n^n) - 1 < 1$.
\end{proof}

Recall that $\mathcal{X}_2$ denotes the set of all Borel measures $\xi$ on $\mathbb{R}^2$ such that $\xi(A) \in \{0, 1, 2, \ldots \}$ for any  bounded Borel set $A$ in $\mathbb{R}^2$.
\begin{definition}\label{D48}
For any $\xi \in \mathcal{X}_2$, we define the LIS of $\xi$ as follows,
\begin{align*}
\lis(\xi) \coloneqq &\max\{k : \ \exists\ (x_1, y_1), (x_2, y_2),\ldots, (x_k, y_k) \in \mathbb{R}^2 \text{ such that }\\
   &\ \ \xi(\{(x_i, y_i)\})\ge 1, \ \forall\ i \in [k] \text{ and } (x_i - x_j)(y_i - y_j)>0, \ \forall\ i \neq j \}.
\end{align*}
\end{definition}

It is easily seen that the function $\lis(\cdot)$ is non-decreasing on $\mathcal{X}_2$ in the sense that, if $\xi, \zeta \in \mathcal{X}_2$ with $\xi \le \zeta$, we have $\lis(\xi) \le \lis(\zeta)$. Moreover, for any $n$ points $\{(x_i, y_i)\}_{i = 1}^n$ in $\mathbb{R}^2$ such that $x_i \neq x_j$ and $y_i \neq y_j$ for all $i \neq j$, define the integer-valued measure $\xi$ as follows,
\[
\xi(A) \coloneqq \sum_{i = 1}^n \mathds{1}_A((x_i, y_i)), \quad \forall A \in \mathcal{B}(\mathbb{R}^2).
\]
Then we have $\lis(\xi) = \lis(\{(x_i, y_i)\}_{i = 1}^n)$, where the latter one is defined in Definition \ref{D3}.

\begin{lemma}\label{L48}
Let $(V_1, \ldots, V_n)$ be a random vector which has the same distribution as $(Y_1^{(n)}, \ldots, Y_n^{(n)})$. For any $m \le n$ and any increasing indices $\bs{b} = (b_1, \ldots, b_m)$, define the point process $\xi_{\bs{b}}^{(n)}$ as in the previous two lemmas, that is,
\[
\xi_{\bs{b}}^{(n)}(A) \coloneqq \sum_{i = 1}^{m}\mathds{1}_A((i, V_{b_i})), \quad \forall A \in \mathcal{B}(\mathbb{N}\times[0, 1]).
\]
Then $\lis(\xi_{\bs{b}}^{(n)})$ and $\lis(\pi_{\bs{b}})$ have the same distribution, where $\pi \sim \n$.
\end{lemma}
\begin{proof}
By the remarks after Definition \ref{D41}, we have
\[
\Phi(\{(i, V_{b_i})\}_{i = 1}^m ) = \Phi((V_{b_1}, V_{b_2}, \ldots, V_{b_m})) = \Phi((V_1, V_2, \ldots, V_n))_{\bs{b}},
\]
where $\Phi((V_1, V_2, \ldots, V_n))$ in the last term has the distribution $\n$.
The lemma follows by the fact that
\[
\lis(\xi_{\bs{b}}^{(n)}) = \lis(\{(i, V_{b_i})\}_{i = 1}^m) = \lis(\Phi(\{(i, V_{b_i})\}_{i = 1}^m )). \qedhere
\]
\end{proof}

Now we are in the position to prove Lemma \ref{L40}. In the following, we use $\lambda_n$ to denote the uniform measure on $S_n$.
\begin{proof}[Proof of Lemma \ref{L40}]
The lemma can be divided into two parts. For the first part, we show that, for any $\epsilon > 0$,
\begin{equation}
\lim_{n \to \infty} \max_{\bs{b} \in Q(n, k_n)} \n\left(\pi \in S_n : \frac{\lis(\pi_{\bs{b}})}{\sqrt{k_n}} \le 2 e^{\frac{-|\beta|}{2}} - \epsilon \right) = 0 .  \label{eq:L41b}
\end{equation}
Given $n > 0$, for any $\bs{b} \in Q(n, k_n)$, by Lemma \ref{L48}, $\lis(\xi_{\bs{b}}^{(n)})$ and $\lis(\pi_{\bs{b}})$ have the same distribution, where $\xi_{\bs{b}}^{(n)}$ is the point process defined in that lemma. Moreover, by Lemma \ref{L45}, there exists a point process $\eta_{k_n}$ such that $\eta_{k_n} \le \xi_{\bs{b}}^{(n)}$ almost surely and $\eta_{k_n}$ is defined by
\begin{equation}
\eta_{k_n}(A) \coloneqq \sum_{i = 1}^{k_n}B_{n, i}\cdot\mathds{1}_{A}((i, U_i)) \quad \forall A\in \mathcal{B}(\mathbb{N}\times[0, 1]), \label{eq:L41c}
\end{equation}
where $\{U_i\}_{i =1}^{k_n} \cup \{B_{n, i}\}_{i = 1}^{k_n}$ are $2k_n$ independent random variables with each $U_i$ being uniformly distributed on $[0, 1]$ and each $B_{n, i}$ being a Bernoulli random variable with $\p(B_{n, i} = 1) = \min(q_n^n, 1/q_n^n)$.
Hence, by the monotonicity of $\lis(\cdot)$ on $\mathcal{X}_2$, we have
\begin{align*}
\n\left(\pi \in S_n : \frac{\lis(\pi_{\bs{b}})}{\sqrt{k_n}} \le 2 e^{\frac{-|\beta|}{2}} - \epsilon \right) &=
\p\left(\frac{\lis(\xi_{\bs{b}}^{(n)})}{\sqrt{k_n}} \le 2 e^{\frac{-|\beta|}{2}} - \epsilon \right)\\
&\le \p\left(\frac{\lis(\eta_{k_n})}{\sqrt{k_n}} \le 2 e^{\frac{-|\beta|}{2}} - \epsilon \right).
\end{align*}
We complete the proof of (\ref{eq:L41b}) by showing that,
\begin{equation}
\lim_{n \to \infty} \p\left(\frac{\lis(\eta_{k_n})}{\sqrt{k_n}} > 2 e^{\frac{-|\beta|}{2}} - \epsilon \right) = 1 , \label{eq:L41d}
\end{equation}
for any $\epsilon > 0$. First we show that
\begin{equation}
\lim_{n \to \infty} \min(q_n^n, 1/q_n^n) =  e^{-|\beta|}. \label{eq:L40a}
\end{equation}
Assuming $0 < q_n \le 1$, since $\lim_{n \to \infty}n(1-q_n) = \beta$ and $\lim_{n \to \infty}\frac{\ln{q_n}}{q_n - 1} = 1$, we have
\[
\lim_{n \to \infty} q_n^n = \lim_{n \to \infty} e^{n\ln{q_n}} = \lim_{n \to \infty}e^{n(q_n - 1)} = e^{-|\beta|}.
\]
The case $q_n > 1$ can be shown similarly. Hence, by (\ref{eq:L40a}), for any $\epsilon_1>0$, there exists $N_1 > 0$ such that, for any $n > N_1$, we have $\min(q_n^n, 1/q_n^n) > e^{-|\beta| - \epsilon_1}$. Thus, by the law of large numbers and the fact that $\lim_{n \to \infty}k_n = \infty$, we have
\begin{equation}
\lim_{n \to \infty} \p\left(\textstyle\sum_{i = 1}^{k_n} B_{n, i} > k_n e^{-|\beta| - \epsilon_1}\right) = 1. \label{eq:L41e}
\end{equation}
Given $\bs{U}=(U_1, \ldots, U_{k_n})$ and $\bs{B}=(B_{n, 1}, \ldots, B_{n, k_n})$, let $\Lambda(\bs{U},\bs{B})$ denote the set of points in $\mathbb{R}^{2}$ defined by
\[
\Lambda(\bs{U},\bs{B}) \coloneqq \{(i, U_i) : i \in [k_n] \text{ and } B_{n, i} = 1\}.
\]
By the definition of $\eta_{k_n}$ and Definition \ref{D48}, we have
\[
\lis(\eta_{k_n}) = \lis(\Lambda(\bs{U},\bs{B})).
\]
Moreover, conditioned on $\sum_{i = 1}^{k_n} B_{n, i} = m$, by the independence of $\bs{U}$ and $\bs{B}$, it is easily seen that $\lis(\Lambda(\bs{U},\bs{B}))$ has the same distribution as $\lis(\pi)$ with $\pi \sim \lambda_m$.
For any $0 < \epsilon_2, \epsilon_3 < 1$, by the result of Vershik and Kerov \cite{VK}, there exists $M > 0$ such that, for any $m > M$,
\begin{equation}
\lambda_m\left(\frac{\lis(\pi)}{\sqrt{m}} > 2 - \epsilon_2\right) > 1 - \epsilon_3. \label{eq:L41f}
\end{equation}
Since $\lim_{n \to \infty}k_n = \infty$ and (\ref{eq:L41e}), there exists $N > N_1$ such that, for any $n > N$, we have
\[
k_n e^{-|\beta| - \epsilon_1} > M \quad \text{ and } \quad \p\left(\textstyle\sum_{i = 1}^{k_n} B_{n, i} > k_n e^{-|\beta| - \epsilon_1}\right) > 1 - \epsilon_3.
\]
Let $E_m$ denote the event $\{\sum_{i = 1}^{k_n} B_{n, i} = m\}$ and $s \coloneqq \lfloor k_n e^{-|\beta| - \epsilon_1} \rfloor + 1$. For any $n > N$, we have
\begin{align*}
&\textstyle\p\Big(\lis(\eta_{k_n}) > (2 - \epsilon_2) \sqrt{k_n e^{-|\beta| - \epsilon_1}}\Big)\\
 \ge \ &\textstyle\sum\limits_{m = s}^{k_n}
 \p\Big(\lis(\eta_{k_n}) > (2 - \epsilon_2) \sqrt{k_n e^{-|\beta| - \epsilon_1}}\ \Big|\ E_m \Big)\cdot
 \p(E_m)\\
 \ge \ &\textstyle\sum\limits_{m = s}^{k_n}
 \p\Big(\lis(\eta_{k_n}) > (2 - \epsilon_2) \sqrt{m}\ \Big|\ E_m \Big)\cdot\p( E_m )\\
 = \ &\textstyle\sum\limits_{m = s}^{k_n}\lambda_m\Big(\lis(\pi) > (2 - \epsilon_2) \sqrt{m}\Big)\cdot\p( E_m )\\
 > \ &\textstyle(1 - \epsilon_3)\sum\limits_{m = s}^{k_n} \p( E_m )\\
 = \ &\textstyle(1 - \epsilon_3)\ \p\Big(\sum\limits_{i = 1}^{k_n} B_{n, i} > k_n e^{-|\beta| - \epsilon_1}\Big)\\
 > \ &\textstyle(1 - \epsilon_3)^2.
\end{align*}
The second inequality above follows since $m \ge s \ge k_n e^{-|\beta| - \epsilon_1}$. The third inequailty follows from (\ref{eq:L41f}) and the fact that $m \ge k_n e^{-|\beta| - \epsilon_1} > M$. Therefore, we have shown that $\lim_{n \to \infty}\p\Big(\lis(\eta_{k_n}) > (2 - \epsilon_2) \sqrt{k_n e^{-|\beta| - \epsilon_1}}\Big) = 1$, and (\ref{eq:L41d}) follows from the fact that, by choosing $\epsilon_1$ and $\epsilon_2$ small enough, $(2 - \epsilon_2) \sqrt{e^{-|\beta| - \epsilon_1}}$ can be arbitrarily close to $2 e^{\frac{-|\beta|}{2}}$.\\

For the second part, we need to show that, for any $\epsilon > 0$,
\begin{equation}
\lim_{n \to \infty} \max_{\bs{b} \in Q(n, k_n)} \n\left(\pi \in S_n : \frac{\lis(\pi_{\bs{b}})}{\sqrt{k_n}} \ge 2 e^{\frac{|\beta|}{2}}+\epsilon\right) = 0.          \label{eq:L42b}
\end{equation}
Similar to the proof of (\ref{eq:L40a}), we can show that
\begin{equation}\label{eq:L40b}
\lim_{n \to \infty}\max\left( q_n^n, 1/q_n^n \right) = e^{|\beta|} < 2.
\end{equation}
The last inequality follows since $|\beta| < \ln{2}$. Thus, for any $0 < \epsilon_1 < \ln{2} - |\beta|$, there exists $N_1 >0$ such that, for all $n > N_1$, we have
\[
\max(q_n^n, 1/q_n^n) < e^{|\beta| + \epsilon_1} < 2.
\]
Given $n > N_1$, for any $\bs{b} \in Q(n, k_n)$, by Lemma \ref{L48}, $\lis(\xi_{\bs{b}}^{(n)})$ and $\lis(\pi_{\bs{b}})$ have the same distribution, where $\xi_{\bs{b}}^{(n)}$ is the point process defined in that lemma. Moreover, by Lemma \ref{L46}, there exists a point process $\zeta_{k_n}$ such that $\xi_{\bs{b}}^{(n)} \le \zeta_{k_n}$ almost surely and $\zeta_{k_n}$ is defined by
\begin{equation}
\zeta_{k_n}(A) \coloneqq \sum_{i = 1}^{k_n}\mathds{1}_A((i, U'_i)) + B_{n,i}\cdot\mathds{1}_A((i, U_i)) \quad \forall A\in \mathcal{B}(\mathbb{N}\times[0, 1]), \label{eq:L42c}
\end{equation}
where $\{U_i\}_{i =1}^{k_n}\cup\{U'_i\}_{i =1}^{k_n} \cup \{B_{n, i}\}_{i = 1}^{k_n}$ are $3k_n$ independent random variables with each $U_i, U'_i$ being uniformly distributed on $[0, 1]$ and each $B_{n, i}$ being a Bernoulli random variable with $\p(B_{n, i} = 1) = \max(q_n^n, 1/q_n^n) - 1$.
Hence, by the monotonicity of $\lis(\cdot)$ on $\mathcal{X}_2$, we have
\[
\n\left(\pi \in S_n : \frac{\lis(\pi_{\bs{b}})}{\sqrt{k_n}} \ge 2 e^{\frac{|\beta|}{2}} + \epsilon \right) 
\le \p\left(\frac{\lis(\zeta_{k_n})}{\sqrt{k_n}} \ge 2 e^{\frac{|\beta|}{2}} + \epsilon \right).
\]
We complete the proof of (\ref{eq:L42b}) as well as Lemma \ref{L40} by showing that, for any $\epsilon > 0$,
\begin{equation}
\lim_{n \to \infty} \p\left(\frac{\lis(\zeta_{k_n})}{\sqrt{k_n}} < 2 e^{\frac{|\beta|}{2}} + \epsilon \right) = 1.  \label{eq:L42d}
\end{equation}
Since, for all $n > N_1$, we have $\p(B_{n, i} = 1) = \max(q_n^n, 1/q_n^n) - 1< e^{|\beta| + \epsilon_1} - 1$, by the law of large numbers, we get
\begin{equation}
\lim_{n \to \infty} \p\left(\textstyle\sum_{i = 1}^{k_n} B_{n, i} < k_n (e^{|\beta| + \epsilon_1} - 1)\right) = 1 . \label{eq:L42e}
\end{equation}
Given $\bs{U}'=(U'_1, \ldots, U'_{k_n})$, $\bs{U}=(U_1, \ldots, U_{k_n})$ and $\bs{B}=(B_{n, 1}, \ldots, B_{n, k_n})$, let $\Lambda(\bs{U}',\bs{U},\bs{B})$ denote the set of points in $\mathbb{R}^{2}$ defined by
\[
\Lambda(\bs{U}',\bs{U},\bs{B}) \coloneqq \{(i, U_i) : i \in [k_n] \text{ and } B_{n, i} = 1\}\bigcup \{(i, U'_i) : i \in [k_n]\}.
\]
By the definition of $\zeta_{k_n}$ and Definition \ref{D48}, we have
\begin{equation}
\lis(\zeta_{k_n}) = \lis(\Lambda(\bs{U}', \bs{U},\bs{B})). \label{eq:L42f}
\end{equation}
Based on $\bs{U}'$, $\bs{U}$ and $\bs{B}$, define another set of points in $\mathbb{R}^{2}$ as follows,
\[
\Lambda^{+}(\bs{U}',\bs{U},\bs{B}) \coloneqq \{(i+1/2, U_i) : i \in [k_n] \text{ and } B_{n, i} = 1\}\bigcup \{(i, U'_i) : i \in [k_n]\}.
\]
Then, we have
\begin{equation}
\lis(\Lambda(\bs{U}', \bs{U},\bs{B})) \le \lis(\Lambda^{+}(\bs{U}', \bs{U},\bs{B})). \label{eq:L42g}
\end{equation}
Since, by Definition \ref{D3}, no two points with the same $x$ coordinate can be both within an increasing subsequence, by increasing the $x$ coordinates of those points in $\Lambda(\bs{U}', \bs{U},\bs{B})$ which reside on the same vertical line as other points by $1/2$, the relative ordering of the shifted point with other points does not change, except the one which has the same $x$ coordinate when unshifted. Combining (\ref{eq:L42f}) and (\ref{eq:L42g}), we have
\begin{equation}
\lis(\zeta_{k_n})  \le \lis(\Lambda^{+}(\bs{U}', \bs{U},\bs{B})). \label{eq:L42h}
\end{equation}
Moreover, conditioned on $\sum_{i = 1}^{k_n} B_{n, i} = m$, by independence of $\bs{U}', \bs{U}$ and $\bs{B}$, it is easily seen that $\lis(\Lambda^{+}(\bs{U}',\bs{U},\bs{B}))$ has the same distribution as $\lis(\pi)$ with $\pi \sim \lambda_{k_n+m}$.
For any $0 < \epsilon_2, \epsilon_3 < 1$, by the result of Vershik and Kerov \cite{VK} again, there exists $M > 0$ such that, for any $k > M$,
\[
\lambda_k\left(\frac{\lis(\pi)}{\sqrt{k}} < 2 + \epsilon_2\right) > 1 - \epsilon_3.
\]
Since $\lim_{n \to \infty}k_n = \infty$ and (\ref{eq:L42e}), there exists $N > N_1$ such that, for any $n > N$, we have
\[
k_n > M \quad \text{ and } \quad
\p\left(\textstyle\sum_{i = 1}^{k_n} B_{n, i} < k_n (e^{|\beta| + \epsilon_1} - 1)\right) > 1 - \epsilon_3.
\]
Let $s \coloneqq \lceil k_n (e^{|\beta| + \epsilon_1} - 1) \rceil - 1$. Recall that $E_m$ denotes the event $\{\sum_{i = 1}^{k_n} B_{n, i} = m\}$. For any $n > N$, we have
\begin{align*}
&\textstyle\p\Big(\lis(\zeta_{k_n}) < (2 + \epsilon_2) \sqrt{k_n e^{|\beta| + \epsilon_1}}\Big)\\
 \ge \ &\textstyle\sum\limits_{m = 0}^{s}
 \p\Big(\lis(\zeta_{k_n}) < (2 + \epsilon_2) \sqrt{k_n e^{|\beta| + \epsilon_1}}\ \Big|\ E_m \Big)
 \cdot\p(E_m)\\
 \ge \ &\textstyle\sum\limits_{m = 0}^{s}
 \p\Big(\lis(\zeta_{k_n}) < (2 + \epsilon_2) \sqrt{k_n + m}\ \Big|\ E_m \Big)\cdot\p(E_m)\\
 \ge \ &\textstyle\sum\limits_{m = 0}^{s} \p\Big(\lis(\Lambda^{+}(\bs{U}', \bs{U},\bs{B})) < (2 + \epsilon_2) \sqrt{k_n + m}\ \Big|\ E_m \Big) \cdot\p(E_m)\\
 = \ &\textstyle\sum\limits_{m = 0}^{s}\lambda_{k_n + m}\Big(\lis(\pi) < (2 + \epsilon_2) \sqrt{k_n + m}\Big)\cdot\p(E_m)\\
 > \ &\textstyle(1 - \epsilon_3)\sum\limits_{m = 0}^{s}\p(E_m)\\
 = \ &\textstyle(1 - \epsilon_3)\ \p\Big(\sum\limits_{i = 1}^{k_n} B_{n, i} < k_n (e^{|\beta| + \epsilon_1} - 1)\Big)\\
 > \ &\textstyle(1 - \epsilon_3)^2.
\end{align*}
The second inequality follows because
\[
k_n + m \le k_n +s < k_n + k_n (e^{|\beta| + \epsilon_1} - 1) = k_n e^{|\beta| + \epsilon_1},
\]
and the third inequality follows from (\ref{eq:L42h}).
Therefore, we have shown that $\lim_{n \to \infty}\p\Big(\lis(\zeta_{k_n}) < (2 + \epsilon_2) \sqrt{k_n e^{|\beta| + \epsilon_1}}\Big) = 1$, and (\ref{eq:L42d}) follows from the fact that, by choosing $\epsilon_1$ and $\epsilon_2$ small enough, $(2 + \epsilon_2) \sqrt{e^{|\beta| + \epsilon_1}}$ can be arbitrarily close to $2 e^{\frac{|\beta|}{2}}$.

\end{proof}

\section{Discussion and open questions}
\begin{itemize}
\item[1.] Consider the partially ordered set $(S_n, \le_L)$, we conjecture the following stochastic dominance of Mallows measure: For any $0< q < q'$, we have $\m \preceq \mu_{n, q'}$, i.e.\,$\m$ is stochastically dominated by $\mu_{n, q'}$. By Strassen's theorem (cf.\,\cite{lindvall}), the conjecture is equivalent to the following statement: there exists a coupling $(X, Y)$ with $X \sim \m$ and $Y 
\sim \mu_{n, q'}$ such that $X \le_L Y$.
\item[2.] In the proof of Corollary \ref{C92}, we show that $\bar{J}$ defined in Theorem \ref{M3} is attained when $\phi(x) = x$ given that $\lim_{n\to\infty} n (1-q_n) = \lim_{n\to\infty} n (1-q'_n) = \beta$. In fact, for any $\beta \in \mathbb{R}$, if $\gamma = 0, \beta, \pm\infty$, taking $\phi(x)$ to be the diagonal of the unit square gives the supremum of the following variational problem
\[
\sup_{\phi \in B^1_{\nearrow}} \int_0^1 \sqrt{\dot{\phi}(x) \rho(x, \phi(x))} \, dx.
\]
Note that, when $\gamma = \pm\infty$, we extend the definition of $\rho(x, y, \beta, \gamma)$ as follows, (We explicitly add $\beta, \gamma$ as the argument of the density $\rho$.)
\[
\rho(x, y, \beta, \pm\infty) \coloneqq \lim_{\gamma \to \pm\infty} \rho(x, y, \beta, \gamma) = \lim_{\gamma \to \pm\infty}\int_0^1 u(x, t, \beta)\cdot u(t, y, \gamma)\,dt.
\]
In fact, it is not hard to show that the above limits exist with
\[
\rho(x, y, \beta, \infty) = u(x, y, \beta) \quad\text{ and }\quad \rho(x, y, \beta, -\infty) = u(x, y, -\beta).
\]
It is unknown to us whether $\phi(x) = x$ solves the above variational problem for arbitrary $\beta, \gamma \in \mathbb{R}$.

\end{itemize}
\section*{Acknowledgement}
I am grateful to my supervisor Nayantara Bhatnagar for her helpful advice and suggestions during the research as well as her guidance in the completion of this paper. I also thank an anonymous referee for pointing out an incorrect proof in the first draft of this paper as well as giving me detailed typographical and grammatical suggestions.
\bibliographystyle{imsart-nameyear}
\bibliography{sampart}

\end{document}